\documentclass[reqno]{amsart}
\usepackage{amssymb}
\usepackage{amsfonts}
\usepackage{amsmath}
\usepackage{amsthm, amsmath, amssymb, enumerate}
\usepackage{amssymb, enumitem}
\usepackage{bbm}
\usepackage{amscd}
\usepackage[all]{xy}
\usepackage[hidelinks]{hyperref}
\usepackage{amsmath}
\usepackage{amssymb}
\usepackage{amsthm}
\usepackage[sort&compress,numbers]{natbib}
\usepackage[left=2cm,right=2cm,bottom=3cm,top=3cm]{geometry}
\usepackage{tikz}
\usepackage{tikz-cd}
\usepackage{wrapfig}
\usepackage{epigraph}
\usepackage{dirtytalk}
\usepackage{csquotes}

\newtheorem{theorem}{Theorem}[section]

\newtheorem{proposition}[theorem]{Proposition}
\newtheorem{lemma}[theorem]{Lemma}
\newtheorem{corollary}[theorem]{Corollary}
\theoremstyle{definition} 
\newtheorem*{claim*}{Claim}
\newtheorem{definition}[theorem]{Definition}

\newtheorem{remark}[theorem]{Remark}

\AtBeginDocument{   \def\MR#1{}}

\begin{document}
\title[Polarities, voltages, and capacitors]{Polarities, voltages, and
capacitors:\\
a categorical approach to hulls, envelopes, and completions}
\author{Ivan Di Liberti}
\address{Department of Philosophy, Linguistics and Theory of Science,
University of Gothenburg, Gothenburg, Sweden.}
\email{diliberti.math@gmail.com}
\urladdr{http://www.lupini.org/}
\author{Martino Lupini}
\address{Dipartimento di Matematica, Universit\`{a} di Bologna, Piazza di
Porta S. Donato, 5, 40126 Bologna,\ Italy}
\email{martino.lupini@unibo.it}
\urladdr{http://www.lupini.org/}
\thanks{The authors were partially supported by the Starting Grant 101077154
\textquotedblleft Definable Algebraic Topology\textquotedblright\ from the
European Research Council, the Gruppo Nazionale per le Strutture Algebriche,
Geometriche e le loro Applicazioni (GNSAGA) of the Istituto Nazionale di
Alta Matematica (INDAM), and the University of Bologna. Part of this work
was done during a visit of the M.L. to Chalmers University of Technology and
the University of G\"{o}thenburg. The hospitality of these institutions is
gratefully acknowledged.}
\date{\today }

\begin{abstract}
This article provides a general framework in the context of category theory
where one can recognize as particular instances of the same abstract
construction several notions of completion, envelope, and hull, such as the
Boolean algebra completion of a Boolean algebra, the Dedekind--MacNeille
completion of an ordered set, the multiplier ring of a ring, the multiplier
algebra and the von Neumann envelope of a C*-algebra.

Towards our goal, we lay the foundations of \emph{polarized }category
theory, which is a refinement of classical category theory where categories
are endowed with two distinguished classes of \emph{positive }and \emph{%
negative }arrows. We define in this context the notion of \emph{polarity},
and \emph{voltage}. We explain how a voltage can be created through a \emph{%
capacitor}, which is essentially a polarized version of the notion of
reflective subcategory. In particular, this produces a \emph{completion
functor} (which in the classical case is just the reflector) which assigns
to each object its completion or hull. These applies even when the
completion is not (and cannot be) given by a functor on the whole category,
as it is most often the case.

In this framework, we obtain a general theorem ensuring the existence and
uniqueness of a functorial completion functor. The corresponding completion
of each object is characterized by its two universal properties with respect
to positive and negative arrows.
\end{abstract}

\maketitle


\begin{displayquote}
\textit{Danger! High voltage}
\end{displayquote}

\setcounter{tocdepth}{1}
\tableofcontents
\section{Introduction}

One of the main motivations for the development of category theory, from its
very inception in the work of Eilenberg and Mac Lane \cite%
{eilenberg_general_1945}, is to provide a framework where notions and
constructions in mathematics can be appropriately formalized and understood.
In the case of Eilenberg and Mac Lane, the notion under consideration was
the \emph{naturality }of the exact sequence in the Universal Coefficient
Theorem relating homology and cohomology of spaces \cite%
{eilenberg_group_1942}. The formalization of this notion led them to the
introduction of \emph{categories}, \emph{functors}, and \emph{natural
transformations}; see \cite[Chapter 2]{kromer_tool_2007} for more historical
information on the origins of category theory. From these early days,
category theory has had a host of applications across mathematics (not to
mention computer science), as many of its notions and constructions have
been recognized as instances of categories, functors, natural
transformations, and other fundamental categorical notions.

More recent research in category theory has focussed on further refining and
generalizing its tools, to make them applicable in contexts where they could
not be applied in a straightforward manner. For example, \emph{higher
category theory }and $\infty $-\emph{category theory} \cite%
{lurie_higher_2009}\emph{\ }became necessary to deal with contexts where
associativity and other identities only hold \emph{essentially}, i.e., up to
isomorphism or other higher order coherence requirements. Likewise, \emph{%
enriched category theory} \cite{kelly_basic_2005} provides a context where
hom-sets are endowed with structure, rendering themselves objects in a
category. (It is not coincidental that these two frameworks are intimately
connected.)

This work has a similar motivation, aiming at providing a framework where
important constructions from algebra and functional analysis, which do not
fit the traditional mold of categories--functors--natural transformations,
can be properly treated. These constructions include several \emph{%
completions},\emph{\ hulls}, \emph{envelopes} (injective or otherwise) and
other \emph{canonical extensions}.

These arise in a number of areas including algebra \cite%
{martinez_hull_2002,emmanouil_flat_2011,day_injectivity_1972,guo_relative_2024,sharpe_injective_1972,matlis_injective_1959,facchini_generalized_1994,matlis_injective_1959}%
, (pointfree) topology \cite%
{banaschewski_booleanization_1996,bhattacharjee_hull_2024,hadwin_injectivity_2011,raphael_essential_2005,lang_injective_2013,raphael_epimorphic_2000}%
, functional analysis \cite%
{oikhberg_injectivity_2018,cohen_injective_1964,isbell_injective_1964,lacey_injective_1969,isbell_three_1969}%
, operator algebras \cite%
{choi_injectivity_1977,ruan_injectivity_1989,hamana_injective_1979,hamana_injective_1979-1,hamana_injective_1978,hamana_injective_1985}%
, dynamical systems \cite%
{bryder_boundaries_2017,hamana_injective_2011,hamana_injective_1992}, and
the theory ordered algebraic structures \cite%
{martinez_hull_2002,conrad_essential_1971,anderson_essential_1979,hager_hulls_1999,carrera_hull_2011,rump_essential_2009}%
.

Examples include the \emph{completion }of a Boolean algebra, the MacNeille
completion of a poset, the lateral completion of an abelian lattice-ordered
group \cite{conrad_lateral_1969,rump_lateral_2009,bernau_lateral_1975}, the
essential closure of a non-Archimedean lattice-ordered group \cite%
{bernau_unique_1965,conrad_essential_1971}, the multiplier ring of a ring 
\cite{dauns_multiplier_1969}, the multiplier algebra of a C*-algebra \cite%
{busby_double_1968}, the enveloping von Neumann algebra of a C*-algebra \cite%
{pedersen_algebras_1979}, the injective envelope of a Banach space \cite%
{cohen_injective_1964}, of an operator system \cite{hamana_injective_1979},
of an operator space \cite{ruan_injectivity_1989}, of a C*-algebra \cite%
{hamana_injective_1979}, of a dynamical system \cite{hamana_injective_2011}.

It is well-known that these constructions are typically not functorial, and
the paper \cite{adamek_injective_2002} is devoted to present a large number
of instances of this phenomenon. Our approach stems from the observation
that these constructions, while generally not functorial, do give rise to
functors, one covariant and the other contravariant, when restricted to
suitable subcategories of the original category. This leads us to the
introduction of the fundamental concept of \emph{polarity}, which is simply
a category endowed with two distinguished subcategories, the positive charge
and the negative charge, subject to some natural requirements. A \emph{%
functor} between polarities will then be a pair of functors, one positive
and one negative, each defined on arrows of a given sign, and which can be
either covariant or contravariant. \emph{Natural transformations} between
functors also admit an obvious bipolar generalization. The classical triad
of categories--functors--natural transformations can be seen as a particular
instance of this framework, when all the arrows have the same charge. In
terms of these fundamental notions, we define \emph{voltages}: polarities
endowed with a distinguished endofunctor. (For comparison, monads are
structures of the same kind, being defined in terms of a category with an
endofunctor and natural transformations satisfying certain axioms \cite[%
Chapter 10]{awodey_category_2006}.) We also explain how polarities can be
seen as a particular instance of categories \emph{enriched }over the
monoidal category of polar sets (which are simply sets with two
distinguished subsets of \emph{positive }and \emph{negative }elements,
respectively). Likewise, functors between polarities, being either positive
or negative, can be seen as objects of a category enriched over \emph{signed
sets }(sets endowed with a binary \emph{partition}).

In this context we define the notion of \emph{capacitor}, which is a
forgetful\ functor between categories subject to certain axioms. These
axioms ensure the existence of a (partially defined) left adjoint, which is
the desired \emph{completion functor }associated with the capacitor, and
turns the coarser category into a voltage. In this case, we say that the
forgetful functor \emph{creates} hulls. Inclusion functors between
reflective subcategories can be seen as a particular case, and correspond to
the (rare) case in which the left adjoint exists \textquotedblleft
globally\textquotedblright . We will show that even in this more general
context one can always obtain a \emph{reflector}, although again only
partially defined.

As an upshot of our analysis, we also clarify the \emph{universal properties 
}that characterize the completion of a given object. There are indeed two of
such properties, one covariant and one contravariant, with respect to
positive and negative arrows, respectively. This is the fundamental reason
behind the introduction of polarities.

We conclude with explaining how the Boolean algebra completion of a Boolean
algebra, the MacNeille completion of a poset, the multiplier ring of a ring
(with local units), the multiplier algebra and the von Neumann envelope of a
C*-algebra are all \emph{created }by certain canonical forgetful functors.
For example, in the case of the MacNeille completion of a poset, one simply
considers the forgetful functor from the category of complete lattices and 
\emph{continuous }monotone maps to the category of posets and monotone maps.
It is not clear whether there exist examples of \textquotedblleft
hulls\textquotedblright\ which have been considered in mathematics that can 
\emph{not }be seen as particular cases of the general construction described
in this paper. For instance, it would be interesting to consider the
injective hull of modules over a fixed ring.

One could compare this work to Fra\"{\i}ss\'{e} theory in model theory and
functional analysis. Indeed, since the seminal work of Fra\"{\i}ss\'{e} in
the context of model theory \cite{fraisse_sur_1954}---see also \cite%
{hodges_model_1993,hodges_building_1985} and \cite%
{ferenczi_amalgamation_2020,masumoto_jiang-su_2017,ben_yaacov_linear_2014,cantier_fraisse_2024,ghasemi_strongly_2021,bryant_fraisse_2021,tursi_separable_2023,lupini_operator_2015,masumoto_jiang-su_2017,bartosova_ramsey_2017,ben_yaacov_fraisse_2015,eagle_fraisse_2016,lupini_fraisse_2018}
for its extension to metric structures---until its later category-theoretic
reformulations and extensions \cite%
{masulovic_dual_2017,masulovic_categorical_2017,masulovic_pre-adjunctions_2018,rosicky_uniqueness_2019,chen_amalgamable_2019,kubis_fraisse_2014,kubis_lelek_2017,kakol_non-archimedean_2017,kubis_injective_2015,ghasemi_universal_2020,kubis_game-theoretic_2018}%
, one of the main upshots of Fra\"{\i}ss\'{e} theory has been to provide a
unified framework to recognize and study constructions that produce
\textquotedblleft generic\textquotedblright\ objects in a given class. In a
similar spirit, this work offers a unified framework where a number of
constructions of \textquotedblleft completions\textquotedblright ,
\textquotedblleft hulls\textquotedblright , and \textquotedblleft
envelopes\textquotedblright\ can be recognized as particular instances of a
general phenomenon, as they can all be seen, in our terminology, as \emph{%
hulls }that are \emph{created }by a certain \emph{forgetful functors}. The
analogy is close: as the foundational theorem of any Fra\"{\i}ss\'{e}%
--theoretic framework is an Existence and Uniqueness Theorem for the Fra%
\"{\i}ss\'{e} limit (generic object), in the case of completions the
fundamental result is the Existence and Uniqueness\ Theorem for the
completion. Precisely as the Fra\"{\i}ss\'{e} limit is characterized by its
universality (an existence assertion) and ultrahomogeneity (an existence and
uniqueness assertion), the completion is also uniquely characterized by its
universality (again, a uniqueness assertion) and maximality (an existence
and uniqueness assertion).

Our ultimate motivation for considering this notion is to generalize to the
polarized context sites, sheafs, and topoi. This will be the subject of
further work, where the voltages will be shown to provide the right target
for the polarized\ version of sheaves. Other categorical approaches to
completions and hulls can be found in \cite%
{bezhanishvili_dedekind_2013,kubis_injective_2015,bezhanishvili_functorial_2016,hager_minimum_2023,cecco_categorical_2024,hager_holder_2014,banaschewski_categorical_1967}%
.

The rest of this paper is divided into 7 sections. Section \ref%
{Section:voltages} presents the fundamental notions of the polarized\
refinement of category theory: monopole, polarity, voltage. Section \ref%
{Section:nature} presents several examples of naturally occurring
electricategories and polarities. Section \ref{Section:capacitors}
introduces \emph{capacitors}, isolating the key ingredients necessary to
produce a partially defined\ left adjoint to a given forgetful functor,
yielding the desired completion functor and a corresponding voltage.

The last sections are devoted to explaining how fundamental completion
constructions in mathematics can be seen as particular instances of this
general framework: the Boolean algebra completions of a Boolean algebra
(Section \ref{Section:rings}), the MacNeille completion of an ordered sets
(Section \ref{Section:posets}), the multiplier ring of a ring with local
units (Section \ref{Section:rings}), the multiplier algebra of a C*-algebra
(Section \ref{Section:algebras}), and the von Neumann envelope of a
C*-algebra (Section \ref{Section:vN-envelope}). Other examples and
applications will be the content of further work.

\subsection{Acknowledgments}

We are grateful to Luigi Caputi, Eusebio Gardella, Matthew Kennedy, Fosco
Loregian, Giuseppe Metere, Hannes Thiel, and Marvin for many useful
conversations and remarks.

\section{Polarities and voltages\label{Section:voltages}}

In this section we lay out the fundamental notions of a generalization of
category theory that could be called \emph{polarized category theory}. In
this context, arrows can have a (positive or negative) sign, which is
preserved by composition. We explain how the corresponding notion of \emph{%
polarity} can this be seen as a particular instance of \emph{enriched
category} over the semilattice of subsets of $\left\{ +,-\right\} $.

\subsection{Monopoles and polarities}

Let $\mathcal{C}$ be a category. Recall that a \emph{wide }subcategory of $%
\mathcal{C}$ is a subcategory that contains all objects of $\mathcal{C}$.
Let us say that a refinement\emph{\ }of $\mathcal{C}$ is subcategory that
contains all objects of $\mathcal{C}$ as well as all arrows in $\mathcal{C}$
that are isomorphisms.

\begin{definition}
\label{Definition:polarity}We define:

\begin{itemize}
\item a positive\emph{\ monopole} to be a category $\mathcal{C}$ endowed
with a distinguished refinement $\mathcal{C}_{+}$ of \emph{positive }arrows;

\item a negative monopole to be a category endowed with a distinguished
refinement $\mathcal{C}_{-}$ of \emph{negative }arrows;

\item a \emph{polarity }to be a category that is both a positive monopole
and a negative monopole.
\end{itemize}
\end{definition}

In the following we assume a monopole to be positive unless otherwise
specified. It is clear that there is no difference between positive and
negative electricategories, but it is terminologically convenient to
distinguish them.

Every category can be regarded as a polarity, where every arrow is positive
and negative. Conversely, any polarity $\mathcal{C}$ gives rise to a
category $\underline{\mathcal{C}}$ obtained by forgetting the polarization,
i.e., the information about the elementary subcategories $\mathcal{C}_{+}$
and $\mathcal{C}_{-}$ of positive and negative arrows. In this case, we say
that $\mathcal{C}$ is a \emph{polarization }of the category $\underline{%
\mathcal{C}}$. Obviously, this establishes an adjunction between polarities
and categories.

Similar considerations and terminology apply to electricategories. Thus,
every monopole $\mathcal{C}$ has an underlying category $\underline{\mathcal{%
C}}$. We then say that $\mathcal{C}$ is a \emph{positive polarization} of $%
\underline{\mathcal{C}}$.

\begin{definition}
Let $\mathcal{C}$ and $\mathcal{D}$ be monopoles. We say that $\mathcal{C}$
is a \emph{sub-monopole} of $\mathcal{D}$ if $\mathcal{C}$ and $\mathcal{D}$
have the same objects and arrows, and every positive arrow in $\mathcal{C}$
is also positive in $\mathcal{D}$.
\end{definition}

Obviously the collection of sub-monopoles of a fixed monopole $\mathcal{C}$
forms a complete lattice.

\begin{definition}
Let $\mathcal{C}$ and $\mathcal{D}$ be polarities. Then we say that $%
\mathcal{C}$ is a \emph{sub-polarity} of $\mathcal{D}$ if $\mathcal{C}$ and $%
\mathcal{D}$ have the same objects and arrows, and every positive
(respectively, negative) arrow in $\mathcal{C}$ is also a positive
(respectively, negative) arrow in $\mathcal{D}$;
\end{definition}

In a polarity we can define the usual notions from category theory but
restricting to just positive and negative arrows. For instance:

\begin{definition}
\label{Definition:epic}Let $\mathcal{C}$ be a polarity, and $f$ an arrow in $%
\mathcal{C}$. Then $f$ is \emph{negatively epic }if for all \emph{negative }%
arrows $p,q$ in $\mathcal{C}$, $pf=qf$ implies $p=q$.
\end{definition}

If $f$ is negative, then $f$ is negatively epic if and only if $f$ is epic
in the category $\mathcal{C}_{-}$. However, in general one can not
reformulate the definition in this fashion since $f$ is not an arrow in $%
\mathcal{C}_{-}$. Notice also that being negatively epic is more generous
than being epic in $\underline{\mathcal{C}}$.

\subsection{Functors between polarities}

Let $\mathcal{C}$ and $\mathcal{D}$ be two polarities. A positive (polar)
functor $F:\mathcal{C}\rightarrow \mathcal{D}$ is given by functors $F_{+}:%
\mathcal{C}_{+}\rightarrow \mathcal{D}_{+}$ and $F_{-}:\mathcal{C}%
_{-}\rightarrow \mathcal{D}_{-}$ that agree on objects. A negative (polar)
functor $F:\mathcal{C}\rightarrow \mathcal{D}$ is given by functors $F_{+}:%
\mathcal{C}_{+}^{\mathrm{op}}\rightarrow \mathcal{D}_{+}$ and $F_{-}:%
\mathcal{C}_{-}\rightarrow \mathcal{D}_{-}$ that agree on objects. Observe
that such functors can be composed in the obvious way, in such a way that
the sign of the resulting functor obeys the product sign rules. A functor $%
\underline{\mathcal{C}}\rightarrow \underline{\mathcal{D}}$ in particular
yields a positive polar functor $\mathcal{C}\rightarrow \mathcal{D}$ by
restriction.

If $F,G:\mathcal{C}\rightarrow \mathcal{D}$ are positive functors, then a
natural transformation $\eta :F\Rightarrow G$ is given by arrows $\eta
_{x}:Fx\rightarrow Gx$ in $\mathcal{C}$ for $x\in \mathcal{C}$ that define
natural transformations $\eta _{+}:F_{+}\Rightarrow G_{+}$ and $\eta
_{-}:F_{-}\Rightarrow G_{-}$. Explicitly, this means that:

\begin{enumerate}
\item for any arrow $f:A\rightarrow B$ in $\mathcal{C}_{+}$ the diagram%
\begin{equation*}
\begin{array}{ccc}
Fx & \overset{F_{+}f}{\rightarrow } & Fx \\ 
\eta _{A}\downarrow &  & \downarrow \eta _{B} \\ 
Gx & \underset{G_{+}f}{\rightarrow } & Gx%
\end{array}%
\end{equation*}%
commutes, and

\item for any arrow $f:A\rightarrow B$ in $\mathcal{C}_{-}$ the diagram%
\begin{equation*}
\begin{array}{ccc}
Fx & \overset{F_{-}f}{\rightarrow } & Fx \\ 
\eta _{A}\downarrow &  & \downarrow \eta _{B} \\ 
Gx & \underset{G_{-}f}{\rightarrow } & Gx%
\end{array}%
\end{equation*}%
commutes.
\end{enumerate}

If $F:\mathcal{C}\rightarrow \mathcal{D}$ is a \emph{positive} functor and $%
G:\mathcal{C}\rightarrow \mathcal{D}$ is a \emph{negative }functor, then a
natural transformation $\eta :F\Rightarrow G$ is given by arrows $\eta
_{x}:Fx\rightarrow Gx$ in $\mathcal{D}$ for objects $x$ of $\mathcal{C}$
such that:

\begin{enumerate}
\item for any arrow $f:A\rightarrow B$ in $\mathcal{C}_{+}$ the diagram%
\begin{equation*}
\begin{array}{ccc}
Fx & \overset{F_{+}f}{\rightarrow } & Fx \\ 
\eta _{A}\downarrow &  & \downarrow \eta _{B} \\ 
Gx & \underset{G_{+}f}{\leftarrow } & Gx%
\end{array}%
\end{equation*}%
commutes, and

\item for any arrow $f:A\rightarrow B$ in $\mathcal{C}_{-}$ the diagram%
\begin{equation*}
\begin{array}{ccc}
Fx & \overset{F_{-}f}{\rightarrow } & Fx \\ 
\eta _{A}\downarrow &  & \downarrow \eta _{B} \\ 
Gx & \underset{G_{-}f}{\rightarrow } & Gx%
\end{array}%
\end{equation*}%
commutes.
\end{enumerate}

Observe that, again, natural transformations between functors of opposite
sign can be composed in the obvious way, producing a natural transformation
between functors of the same sign.

The notion of functors between monopoles, and their natural transformations
can also be defined analogously.

\subsection{Polarities as enriched categories}

Let $\mathcal{L}$ be the first-order language that comprises two unary
relation symbols $R_{+}$ and $R_{-}$.\ Then a first-order structure in such
a language is just a \emph{polar set}, i.e., a set $A$ endowed with two
subsets $A_{+}$ and $A_{-}$, corresponding to the interpretation of $R_{+}$
and $R_{-}$. A morphism $A\rightarrow B$ between two such structures is a
function mapping $A_{+}$ to $B_{+}$ and $A_{-}$ to $B_{-}$. This defines a
category $\mathbf{Set}_{\pm }$. In this category, the product $A\times B$ of
polar sets $A$ and $B$ is the cartesian product with%
\begin{equation*}
\left( A\times B\right) _{+}:=A_{+}\times B_{+}\text{\quad and\quad }\left(
A\times B\right) _{-}=A_{-}\times B_{-}\text{.}
\end{equation*}%
This monoidal operation and the terminal object $\left\{ \ast \right\} $,
where $\ast $ is both positive and negative, produce a monoidal structure on 
$\mathbf{Set}_{\pm }$ \cite[Section 1.1]{kelly_basic_2005}.

One can then regard every locally small polarity $\mathcal{C}$ as a $\mathbf{%
Set}_{\pm }$-category, i.e., a category enriched over $\mathbf{Set}_{\pm }$ 
\cite[Section 1.2]{kelly_basic_2005}. Indeed, if $x,y$ are objects of $%
\mathcal{C}$, then $\mathrm{Hom}_{\mathcal{C}}\left( x,y\right) $ is a polar
set, where $\mathrm{Hom}_{\mathcal{C}}\left( x,y\right) _{+}$ comprises the
positive arrows and $\mathrm{Hom}_{\mathcal{C}}\left( x,y\right) _{-}$
comprises the negative arrows. Conversely, any $\mathbf{Set}_{\pm }$%
-category can be seen as a polarity. If $\mathcal{C}$ and $\mathcal{D}$ are
polarities, then the positive functors $\mathcal{C}\rightarrow \mathcal{D}$
are precisely the $\mathbf{Set}_{\pm }$-functors \cite[Section 1.2]%
{kelly_basic_2005}.

One can consider the category $\mathbf{Polarity}$ whose objects are the
polarities and whose arrows are the positive or negative functors. As an
arrow in $\mathbf{Polarity}$ is either positive or negative, this yields
what could be called a \emph{signed category}. In a signed category each
arrow has precisely one sign (positive or negative), and the sign of the
composition is the \emph{product }of the signs. Such signed categories can
be thought of as enriched categories over the monoidal category $\pm \mathbf{%
Set}$ of \emph{signed sets}. Such a category has been considered in algebra
and combinatorics in the study of oriented or directed combinatorial and
algebraic structures, such as oriented matroids \cite%
{brini_combinatorics_2005,bollobas_erdos-ko-rado_1997,las_vergnas_convexity_1980}%
.

A signed set is thus a set $A$ endowed with a $\left\{ +,-\right\} $%
-labelled \emph{partition }$\left\{ +A,-A\right\} $. The morphisms again are
the functions preserving the partition. In this category, the \emph{monoidal
operation }is defined by letting $A\otimes B$ to be the cartesian product
with partition defined by%
\begin{equation*}
+\left( A\otimes B\right) :=\left( \left( +A\right) \times \left( +B\right)
\right) \cup \left( \left( -A\right) \times \left( -B\right) \right)
\end{equation*}%
and the monoidal identity is the set $\left\{ +,-\right\} $ with the obvious
partition.

\subsection{Hereditary monopoles}

We now consider a natural property of electricategories.

\begin{definition}
\label{Definition:hereditary}Let $\mathcal{C}$ be a monopole. Then $\mathcal{%
C}$ is left \emph{hereditary} if for all arrows $f,g$ in $\mathcal{C}$:%
\begin{equation*}
gf\text{ positive }\Rightarrow g\text{ positive.}
\end{equation*}
\end{definition}

Let $\mathcal{C}$ be any category. Declaring an arrow to be positive if and
only if it is epic turns $\mathcal{C}$ into a left hereditary monopole.

\begin{definition}
\label{Definition:core}Let $\mathcal{C}$ be a monopole. The \emph{left
hereditary core }$\mathrm{L}\left( \mathcal{C}\right) $ of $\mathcal{C}$ is
the monopole with the same underlying category as $\mathcal{C}$, such that
an $f$ arrow of $\mathcal{C}$ is positive in $\mathrm{L}\left( \mathcal{C}%
\right) $ if and only if:

\begin{enumerate}
\item $f$ is positive in $\mathcal{C}$, and

\item for any other arrow $g$ of $\mathcal{C}$, if $gf$ is defined and
positive in $\mathcal{C}$, then $g$ is positive in $\mathcal{C}$.
\end{enumerate}
\end{definition}

It is easily verified that the left hereditary core is indeed a left
hereditary monopole, and in fact the largest left hereditary sub-monopole of
the given monopole.

The essential core of a monopole has already been considered implicitly in
the literature; see for example \cite%
{porst_characterization_1981,adamek_injective_2002}. Indeed, in the
terminology of \cite{porst_characterization_1981,adamek_injective_2002}, an
arrow $f$ is positive in $\mathrm{L}\left( \mathcal{C}\right) $ if and only
if it is $\mathcal{C}_{+}$-\emph{essential}, where $\mathcal{C}_{+}$ is the
class of positive arrows in $\mathcal{C}$.

\subsection{Voltages}

We define voltages as polarities endowed with additional information.

\begin{definition}
\label{Definition:voltage}A \emph{voltage }is given by:

\begin{itemize}
\item a polarity $\mathcal{C}$;

\item a negative functor $E:\mathcal{C}\rightarrow \underline{\mathcal{C}}$;

\item a natural transformation 
\begin{equation*}
\eta :1_{\underline{\mathcal{C}}}\Rightarrow E
\end{equation*}%
such that 
\begin{equation*}
\eta _{E}:E\Rightarrow EE
\end{equation*}%
is a natural isomorphism.
\end{itemize}
\end{definition}

Notice that in\ Definition \ref{Definition:voltage}, $\underline{\mathcal{C}}
$ is the category obtained from $\mathcal{C}$ by forgetting the
polarization. It is considered as a polarity itself where all the arrows are
positive and negative. Spelling out the definition, a voltage is given by:

\begin{itemize}
\item a category $\underline{\mathcal{C}}$ together with elementary
subcategories $\mathcal{C}_{+}$ and $\mathcal{C}_{-}$;

\item functors 
\begin{equation*}
E_{+}:\mathcal{C}_{+}^{\mathrm{op}}\rightarrow \mathcal{C}\text{\quad
and\quad }E_{-}:\mathcal{C}_{-}\rightarrow \mathcal{C}
\end{equation*}%
agreeing on objects;

\item arrows 
\begin{equation*}
\eta _{A}:A\rightarrow EA
\end{equation*}%
for every object $A$ of $\mathcal{C}$, where 
\begin{equation*}
EA:=E_{+}\left( A\right) =E_{-}\left( A\right) \text{;}
\end{equation*}
\end{itemize}

such that:

\begin{enumerate}
\item $\eta _{EA}:EA\rightarrow EEA$ is an isomorphism for every object $A$
of $\mathcal{C}$;

\item for every arrow $f:A\rightarrow B$ in $\mathcal{C}_{+}$ the diagram%
\begin{equation*}
\begin{array}{ccc}
A & \overset{f}{\rightarrow } & B \\ 
\eta _{A}\downarrow &  & \downarrow \eta _{B} \\ 
EA & \underset{E_{+}f}{\longleftarrow } & EB%
\end{array}%
\end{equation*}%
commutes;

\item for every arrow $f:A\rightarrow B$ in $\mathcal{C}_{-}$ the diagram%
\begin{equation*}
\begin{array}{ccc}
A & \overset{f}{\rightarrow } & B \\ 
\eta _{A}\downarrow &  & \downarrow \eta _{B} \\ 
EA & \underset{E_{-}f}{\rightarrow } & EB%
\end{array}%
\end{equation*}%
commutes.
\end{enumerate}

\begin{remark}
Recall the notion of \emph{dagger category }from \cite[Section 2.3]%
{heunen_categories_2019}. One can regard a dagger category with dagger
functor $\dagger :f\mapsto f^{\dagger }$ as a particular case of \emph{%
voltage}. In the dagger case, all the arrows are both positive and negative.
The contravariant functor $E_{+}$ is the dagger functor, while the functor $%
E_{-}$ is the identity functor. For any object $A$, the arrow $\eta
_{A}:A\rightarrow A$ is the identity.
\end{remark}

\section{Monopoles in nature\label{Section:nature}}

\subsection{Strong morphisms}

Let $\mathcal{C}$ be a category with finite limits and colimits.\ Denote by $%
\mathcal{E}$ the class of epics in $\mathcal{C}$ and by $\mathcal{M}$ the
class of monics in $\mathcal{C}$. Let $a,b$ be arrows in $\mathcal{C}$.
Recall that $a$ is left orthogonal to $b$, and $b$ is right orthogonal to $a$%
, if for any commuting diagram%
\begin{equation*}
\begin{array}{ccc}
\bullet & \overset{a}{\rightarrow } & \bullet \\ 
g\downarrow &  & \downarrow h \\ 
\bullet & \overset{b}{\rightarrow } & \bullet%
\end{array}%
\end{equation*}%
there exists a unique arrow $t$ such that%
\begin{equation*}
at=g\text{\quad and\quad }bt=g\text{.}
\end{equation*}%
Given a collection $\mathcal{H}$ of arrows one let $\mathcal{H}^{\bot }$ be
the collection of arrows that are right orthogonal to all arrows in $%
\mathcal{H}$, and $^{\bot }\mathcal{H}$ be the collection of all arrows that
are left orthogonal to all arrows in $\mathcal{H}$.

A \emph{strong epic }is an arrow in $^{\bot }\mathcal{M}$, i.e., an arrow
that is left orthogonal to all the monics \cite[Definition 1.1]%
{gran_introduction_2021}. Dually, a \emph{strong monic }is an arrow in $%
\mathcal{E}^{\bot }$, i.e., right orthogonal to all epics.

Any strong epic is, in particular, epic. Furthermore, an strong epic that is
also monic must be an isomorphism \cite[Lemma 1.3]{gran_introduction_2021}.
Dually, any strong monic is monic, and it is an isomorphism if and only if
it is also epic.

\subsection{Regular morphisms}

We continue to assume that $\mathcal{C}$ is a category with finite limits
and colimits. We recall the notion of \emph{regular morphism }in $\mathcal{C}
$. Given an arrow $f:a\rightarrow b$, one can consider the pullback diagram 
\begin{equation*}
\begin{array}{ccc}
a\times _{f}a & \rightarrow & a \\ 
\downarrow &  & \downarrow f \\ 
a & \overset{f}{\rightarrow } & b%
\end{array}%
\end{equation*}%
This is called the \emph{kernel pair }$\mathrm{\mathrm{Ker}}\left( f\right) $%
\emph{\ }of $f$. Dually, one can consider the pushout diagram%
\begin{equation*}
\begin{array}{ccc}
a & \overset{f}{\rightarrow } & b \\ 
f\downarrow &  & \downarrow \\ 
b & \rightarrow & b\sqcup _{f}a\text{.}%
\end{array}%
\end{equation*}%
This is called the \emph{cokernel pair }$\mathrm{Coker}\left( f\right) $ of $%
f$.

One defines:

\begin{itemize}
\item the \emph{image }$\mathrm{Im}\left( f\right) $ of $f$ to be the
equalizer of its cokernel pair 
\begin{equation*}
\mathrm{Im}\left( f\right) =\mathrm{Eq}\left( \mathrm{Coker}\left( f\right)
\right) \text{;}
\end{equation*}

\item the \emph{coimage }$\mathrm{Coim}\left( f\right) $ of $f$ to be the
coequalizer of its kernel pair%
\begin{equation*}
\mathrm{CoIm}\left( f\right) =\mathrm{CoEq}\left( \mathrm{\mathrm{Ker}}%
\left( f\right) \right) \text{.}
\end{equation*}
\end{itemize}

By the universal properties of these objects, one obtains a canonical
factorization%
\begin{equation}
\begin{array}{ccc}
a & \overset{f}{\rightarrow } & b \\ 
\downarrow &  & \uparrow \\ 
\mathrm{CoIm}\left( f\right) & \overset{\rho \left( f\right) }{\rightarrow }
& \mathrm{Im}\left( f\right)%
\end{array}%
\text{.\label{Diagram:Factorization}}
\end{equation}%
Let us call the arrow%
\begin{equation*}
\rho \left( f\right) :\mathrm{CoIm}\left( f\right) \rightarrow \mathrm{Im}%
\left( f\right)
\end{equation*}%
the \emph{regular comparison }map; see \cite[Definition 1.1.1]%
{schneiders_quasi-abelian_1999} (in the context of additive categories).
Such an arrow is always both monic and epic \cite[Proposition 1.1.5]%
{schneiders_quasi-abelian_1999}, although it fails in general to be an
isomorphism.

The morphism $f$ is:

\begin{itemize}
\item a \emph{regular epic }if $f=\mathrm{CoIm}\left( f\right) $, and dually

\item a \emph{regular monic }if $f=\mathrm{Im}\left( f\right) $;
\end{itemize}

see \cite[Definition 1.5]{gran_introduction_2021}.

Explicitly, a morphism $f:x\rightarrow y$ is a regular monic if and only if
whenever $g$ is a morphism with target $y$ that equalizes all the pairs of
arrows equalized by $f$ (i.e., $pf=qf\Rightarrow pg=qg$) then $g$ factors
through $f$. Dually, a morphism $f:x\rightarrow y$ is a regular epic every $%
h $ is a morphism with source $x$ that co-equalizes all the pairs of arrows
co-equalized by $f$ (in formulae: $fi=fj\Rightarrow hi=hj$) must factor
through $f$. Notice that these reformulations are meaningful even when $%
\mathcal{C}$ does not have finite limits and colimits.

\begin{remark}
Regular monics and epics are considered in \cite%
{schneiders_quasi-abelian_1999} (in the context of additive categories) and
in \cite{banaschewski_categorical_1967} under the name of \emph{strict}
monics and epics. Both these sources attribute the notion to Grothendieck.
\end{remark}

A morphism is a regular epic if and only if it is the coequalizer of a pair
of arrows; see \cite[Exercise 1.9]{gran_introduction_2021}. Every split epic
is a regular epic, and every regular epic is a strong epic; see \cite[%
Proposition 1.8]{gran_introduction_2021}.

More generally, one can define what it means for an arbitrary morphism to be
regular:

\begin{definition}
\label{Definition:regular}The arrow $f:a\rightarrow b$ is called \emph{%
regular} if the canonical comparison map $\rho \left( f\right) $ is an
isomorphism.
\end{definition}

Notice that the map $\mathrm{Im}\left( f\right) \rightarrow b$ in the
commuting diagram \eqref{Diagram:Factorization} is always a regular monic
and, particularly, a strong monic. Thus, if $f$ is epic then also $\mathrm{Im%
}\left( f\right) \rightarrow b$ is epic, whence an isomorphism. These
observations shows that Definition \ref{Definition:regular} indeed agrees
with the notion of regular epic previously defined. Notice that in general
the composition of regular morphisms need not be regular.

The notion of regular category is defined in terms of the notion of regular
epimorphism; see \cite[Definition 1.10]{gran_introduction_2021}.

\begin{definition}
Let $\mathcal{C}$ be a category with finite limits and colimits. Then $%
\mathcal{C}$ is \emph{regular }if the pullback along any morphism of a
regular epimorphism is a regular epimorphism.
\end{definition}

In a regular category, the factorization of an arrow $f$ as%
\begin{equation*}
f=mq
\end{equation*}%
where $q$ is a regular epimorphism and $m$ is a monomorphism is essentially
unique \cite[Theorem 1.11]{gran_introduction_2021}. Furthermore, such a
factorization is \emph{pullback-stable}, and these properties characterize
regular categories among finitely complete categories \cite[Theorem 1.14]%
{gran_introduction_2021}.

\subsection{Essential monics}

\begin{proposition}
\label{Proposition:regular-monic}Let $\mathcal{C}$ be a category with finite
limits and colimits.

\begin{enumerate}
\item Declaring an arrow in $\mathcal{C}$ to be positive if it is a \emph{%
monic} defines a $\emph{right}$ \emph{hereditary} sub-monopole $\mu \mathcal{%
C}$ of $\mathcal{C}$.

\item Declaring an arrow in $\mathcal{C}$ to be positive if it is a \emph{%
regular monic} defines \emph{right hereditary }sub-monopole of $\mathcal{C}$.
\end{enumerate}
\end{proposition}

\begin{proof}
(1) is obvious. (2) That composition of regular monics is a regular monic is
the content of \cite[Proposition 1.1.7]{schneiders_quasi-abelian_1999} (in
the context of additive categories). That they form a right hereditary class
is the content of \cite[Proposition 1.1.8]{schneiders_quasi-abelian_1999}.
(Recall that \cite{schneiders_quasi-abelian_1999} uses the term strict
rather than regular.)
\end{proof}

Let $\mu \mathcal{C}$ be the monopole as in Proposition \ref%
{Proposition:regular-monic}. One can consider its \emph{left hereditary core 
}$\mathrm{L}\left( \mu \mathcal{C}\right) $ as in Definition \ref%
{Definition:core}. The arrows in $\mathrm{L}\left( \mu \mathcal{C}\right) $
are called \emph{essential monics}. Notice that, by definition, an arrow $f$
is an essential monic if:

\begin{enumerate}
\item it is a monic, and

\item whenever $g$ is an arrow such that $gf$ is a monic, $f$ is a monic.
\end{enumerate}

The notion of \emph{essential regular monic }is defined analogously, where
monics are replaced with \emph{regular }monics.

\subsection{Normal monics}

Let $\mathcal{C}$ be a category with finite limits. Then a (necessarily
monic \cite[Lemma 3.2.2]{borceux_malcev_2004}) morphism $f:X\rightarrow Y$
is \emph{normal }to some equivalence relation $r:R\rightarrow Y\times Y$ on $%
Y$ \cite[Definition 3.2.1]{borceux_malcev_2004} when:

\begin{enumerate}
\item $f\times f$ factors through $m$ yielding a pullback%
\begin{equation*}
\begin{array}{ccc}
X\times X & \overset{\xi }{\rightarrow } & R \\ 
\downarrow &  & \downarrow r \\ 
X\times X & \underset{f\times f}{\rightarrow } & Y\times Y%
\end{array}%
\end{equation*}

\item the diagram%
\begin{equation*}
\begin{array}{ccc}
X\times X & \overset{\xi }{\rightarrow } & R \\ 
p_{0}\downarrow &  & \downarrow d_{0} \\ 
X & \underset{f}{\rightarrow } & Y%
\end{array}%
\end{equation*}%
is also a pullback.
\end{enumerate}

A monic arrow $f$ is normal if it is normal to some equivalence relation $R$%
, which is uniquely determined by $f$ when $\mathcal{C}$ is protomodular
with finite limits \cite[Theorem 3.2.8]{borceux_malcev_2004}. The
composition of two normal monics is a normal monic, and the class of normal
monics is closed by taking pullbacks along arbitrary arrows \cite[%
Proposition 3.2.6(2)]{borceux_malcev_2004}. Suppose now that $\mathcal{E}$
has pushouts. If $s$ is a split monic with splitting $p$, then the $s$ is
normal to the \emph{kernel pair} of $p$. Thus, every split monic is normal.

One can turn $\mathcal{C}$ into a (left hereditary) monopole $\nu \mathcal{C}
$ with \emph{normal monics} as positive morphisms. When $\mathcal{C}$ is 
\emph{semi-abelian category} \cite[Section 3.1]{janelidze_semi-abelian_2002}
the classes of\emph{\ normal monics} and \emph{kernels} coincide \cite[%
Proposition 3.2.20]{borceux_malcev_2004}.

\subsection{Pseudocomplements}

Assume that $\mathcal{A}$ is a \emph{semi-abelian category}. Given an object 
$A$ of $\mathcal{A}$ one defines \emph{ideals }of $A$ to be the subobjects
corresponding to \emph{normal monics }with target $A$. The correspondence
between normal monics and equivalence relations shows that the ordered set $%
\mathrm{Ideal}\left( A\right) $ of ideals of $A$ is in fact a bounded\emph{\
lattice }isomorphic to the lattice of equivalence relations on $A$. (Recall
that a semi-abelian category is, in particular, a regular Mal'cev category.
In any regular Mal'cev category \cite[Proposition 5.1.2]{borceux_malcev_2004}%
, equivalence relations on an object form a lattice \cite[Definition 2.9.1]%
{borceux_malcev_2004}.)

Given ideals $I,J$ of an object $A$, one sets $I\bot J$ if $I\wedge J=0$ in $%
\mathrm{Ideal}\left( A\right) $. Given $I$, the largest ideal $J$ satisfying 
$I\bot J$ is called (if it exists) the \emph{pseudocomplement} of $I$. A
semi-abelian category $\mathcal{A}$ \emph{has pseudocomplements of ideals }%
if for every object $A$ of $\mathcal{A}$, every ideal $I$ of $A$ has the
pseudocomplement, i.e., $\mathrm{Ideal}\left( A\right) $ is a
pseudocomplemented lattice. For example, the semi-abelian categories of 
\emph{C*-algebras} or $\ell $-\emph{groups} have pseudocomplements of
ideals, while the semi-abelian categories of \emph{groups} and \emph{abelian
groups }do not have pseudocomplements of ideals.

\begin{lemma}
\label{Lemma:quotient-kernel}Suppose that $i:X\rightarrow Y$ is a normal
monic in $\mathcal{A}$ that has pseudocomplement $i^{\bot }$ in the lattie
of ideals of $Y$. Define $p:=\mathrm{Coker}\left( i^{\bot }\right)
:Y\rightarrow Z$ and $\xi :=pi$.\ Then $\xi $ is a normal essential monic.
\end{lemma}

\begin{proof}
Since $i$ and $p$ are normal monics, so is $\xi $. Since $i^{\bot }\bot i$,
by the universal property of the cokernel we have that $\mathrm{\mathrm{Ker}}%
\left( \xi \right) =0$ and $\xi $ is monic. Suppose now that $m:W\rightarrow
Z$ is a kernel with $m\bot \xi $. Consider the lift $m^{\prime }$ of $m$
alongside $p$ via pullback, and then the lift $m^{\prime \prime }$ of $%
m^{\prime }$ alongside $i$ via pullback:%
\begin{equation*}
\begin{array}{lllll}
W^{\prime \prime } & \longrightarrow & W^{\prime } & \overset{p^{\prime }}{%
\rightarrow } & W \\ 
m^{\prime \prime }\downarrow &  & m^{\prime }\downarrow &  & m\downarrow \\ 
X & \overset{i}{\longrightarrow } & Y & \overset{p}{\rightarrow } & X%
\end{array}%
\end{equation*}%
Then $m^{\prime \prime }$ is the lift of $w$ alongside $\xi $ by \cite[Lemma
5.8]{awodey_category_2006}. Since $m\bot \xi $ we must have $m^{\prime
\prime }=0$. Thus, $m^{\prime }\bot i$. This shows that $m^{\prime }\leq
i^{\bot }$, i.e., $m^{\prime }$ factors through $i^{\bot }=\mathrm{\mathrm{%
Ker}}\left( p\right) $ and in particular $m^{\prime }p=0$. Thus, $mp^{\prime
}=0$. Since $p$ is a regular epic, and the class of regular epics in a
regular category are pullback-stable, the arrow $p^{\prime }$ is also a
regular epic. Therefore, $m=0$.\ This concludes the proof.
\end{proof}

\begin{lemma}
\label{Lemma:trivial-complement}Suppose that $m:I\rightarrow A$ is a
kernels.\ The following assertions are equivalent:

\begin{enumerate}
\item $m$ is an essential monic;

\item whenever $f:A\rightarrow B$ is such that $fm$ is a normal monic, $f$
is monic;

\item if $g:J\rightarrow A$ is a normal monic with $g\bot m$, then $g=0$.
\end{enumerate}
\end{lemma}

\begin{proof}
(1)$\Rightarrow $(2) is obvious from the definition.

(2)$\Rightarrow $(3)\ Let $g$ be as in (3). Consider the cokernel $%
p:A\rightarrow B$ of $g$.\ Define $f=\mathrm{\mathrm{Ker}}\left( pm\right) $%
. Since $g$ is a normal monic, it must be the kernel of $p$ \cite[%
Proposition 3.2.20]{borceux_malcev_2004}. Since $p\left( mf\right) =0$, we
must have that $mf$ factors through $g$. Since $m\bot g$, this forces $f=0$.
Thus, $pm$ has trivial kernel, and hence it is monic. Since $p$ and $m$ are
normal morphisms, the same holds for $pm$. This shows that $pm$ is a normal
monic. By hypothesis, $p$ must be monic. This implies that $g=\mathrm{%
\mathrm{Ker}}\left( p\right) =0$.

(2)$\Rightarrow $(1) Suppose that $f:A\rightarrow B$ is an arrow such that $%
fm$ is a normal monic. We need to show that $f$ is monic. Let $%
g:J\rightarrow A$ be the kernel of $f$. We need to show that $g=0$. It
suffices to prove that $g\bot m$. Consider the pullback of $m$ and $g$:%
\begin{equation*}
\begin{array}{lllll}
S & \overset{a}{\longrightarrow } & J &  &  \\ 
b\downarrow &  & \downarrow g &  &  \\ 
I & \overset{m}{\longrightarrow } & A & \overset{f}{\rightarrow } & B%
\end{array}%
\end{equation*}%
Then we have that 
\begin{equation*}
0=fga=fmb
\end{equation*}%
Since $fm$ is monic, this implies that $b=0$. Since $g$ is monic and
pullbacks preserve monics \cite[Chapter 5]{awodey_category_2006}, $b$ is
monic, and $S=0$. This shows that $f$ is monic.
\end{proof}

\subsection{Ideals in varieties}

Let $\mathcal{A}$ be a semi-abelian variety of algebras. In this case we can
reformulate the previous definitions as follows:

\begin{itemize}
\item a \emph{monic }in $\mathcal{A}$ is an injective homomorphism;

\item a \emph{normal monic} in $\mathcal{A}$ is an injective homomorphism
whose image is an ideal;

\item a \emph{normal essential monic} in $\mathcal{A}$ is an injective
homomorphism whose image is an \emph{essential }ideal, i.e., an ideal with
nontrivial intersection with any nontrivial ideal;

\item if $I$ and $J$ are ideals of $A$, then 
\begin{equation*}
I\bot J\Leftrightarrow I\cap J=0\text{;}
\end{equation*}

\item if $J$ is an ideal of $A$, then $J$ is essential if and only if for
any ideal $I$ of $A$, if $I\bot J$ then $I=0$;

\item if $I$ is an ideal of $A$, then its pseudocomplement $I^{\bot }$, if
it exists, is an ideal that is largest among the ideals orthogonal to $I$;

\item the variety $\mathcal{A}$ has pseudocomplements of ideals if for any
algebra $A$ in $\mathcal{A}$ and ideal $I$ in $A$, $I$ has a
pseudocomplement.
\end{itemize}

\section{Voltages from capacitors\label{Section:capacitors}}

In this section we introduce a general notion of \emph{hull }of an object of
a monopole. We also define \emph{capacitor} on a monopole. We then explain
how a capacitor gives rise to a canonical \emph{voltage }on the given
monopole.

\subsection{Amphilimits}

We consider a natural notion that combines the existence property in a
colimit with the uniqueness property in a limit. As in the case of usual
limits and colimits, we can reduce to the case of the empty diagram.

\begin{definition}
\label{Definition:amphi-terminal}Let $\mathcal{D}$ be a category. An object $%
z$ of $\mathcal{D}$ is \emph{amphi-terminal} if:

\begin{enumerate}
\item for any object $x$ of $\mathcal{D}$ there exists at least one arrow $%
x\rightarrow z$ in $\mathcal{D}$;

\item for any object $y$ of $\mathcal{D}$ there exists at most one arrow $%
z\rightarrow y$ in $\mathcal{D}$.
\end{enumerate}

Dually, an object is \emph{amphi-initial} if it is amphi-terminal in $%
\mathcal{D}^{\mathrm{op}}$.
\end{definition}

In Definition \ref{Definition:amphi-terminal}, (1) is equivalent to the
request that $z$ be \emph{weakly terminal }\cite[Definition 3.2]%
{kainen_weak_1971}, while (2) is equivalent to the request that $z$ be \emph{%
quasi-initial }\cite[Section 1.4]{hager_holder_2014}. It is obvious that if $%
\mathcal{D}$ is a full subcategory of $\mathcal{D}^{\prime }$ and $z$ is an
object of $\mathcal{D}$ that is amphi-initial in $\mathcal{D}^{\prime }$,
then it is also amphi-initial in $\mathcal{D}$.

It is obvious that an amphi-terminal object, if it exists, is essentially
unique, and has no nontrivial endomorphisms. Conversely, every weakly
terminal object with no nontrivial endomorphisms is amphi-terminal. Notice
that a zero object, i.e., an object that is both initial and terminal, is
also an object that is both amphi-initial and amphi-terminal.

\begin{definition}
\label{Definition:amphi-colimit}Let $F:J\rightarrow \mathcal{D}$ be a
functor. Then an amphi-limit of $F$ is an amphi-terminal object in the
category of cones over $F$.
\end{definition}

\subsection{Completions}

Let $\mathcal{C}$ be a monopole. Fix an object $x$ of $\mathcal{C}$. We
define the \emph{slice monopole }$\mathcal{C}\downarrow x$ to be the
monopole that has as objects the \emph{positive }arrows with target $x$. A
arrow $f\rightarrow g$ in $\mathcal{C}\downarrow x$ is an arrow $\xi $ in $%
\mathcal{C}$ such that $f=g\xi $. Such an arrow is positive if and only if
it is positive as an arrow of $\mathcal{C}$. We also define the co-slice
monopole $x\downarrow \mathcal{C}$ by duality.

\begin{definition}
\label{Definition:hull}Let $\mathcal{C}$ be a monopole, and $x$ be an object
of $\mathcal{C}$.

\begin{itemize}
\item A \emph{completion }or hull $\mathrm{h}_{\mathcal{C}}\left( x\right) $
of $x$ in $\mathcal{C}$ is an amphi-terminal object in (the underlying
category of) $x\downarrow \mathcal{C}$;

\item The object $x$ of $\mathcal{C}$ is \emph{complete} in $\mathcal{C}$ if
its identity arrow is a completion of $x$ in $\mathcal{C}$.
\end{itemize}
\end{definition}

As a particular instance of the corresponding general fact about
amphi-limits, the completion of an object is essentially unique, if it
exists.

\begin{lemma}
\label{Lemma:hull-monic}Let $\mathcal{C}$ be a monopole. Suppose that $\eta
_{x}:x\rightarrow Jx$ is a completion of $x$ in $\mathcal{C}$. Then $\eta
_{x}$ is monic in $\mathcal{C}$.
\end{lemma}

\begin{proof}
Suppose that $g,h:Jx\rightarrow y$ are morphisms in $\mathcal{C}$ such that $%
g\eta _{x}=h\eta _{x}$. Then $g,h$ define two morphisms from $\eta _{x}$ to $%
g\eta _{x}=h\eta _{x}$ in $x\downarrow \mathcal{C}$. By the uniqueness
property of $\eta _{x}$ as the amphi-terminal object in $x\downarrow 
\mathcal{C}$, we must have $g=h$.
\end{proof}

\begin{lemma}
Let $\mathcal{C}$ be a monopole. If $\mu _{x}:x\rightarrow y$ is a
completion of $x$ in $\mathcal{C}$, then $y$ is closed in $\mathcal{C}$.
\end{lemma}

\begin{proof}
We need to show that $1_{y}:y\rightarrow y$ is a completion of $y$. Let $%
f:y\rightarrow a$ be a positive arrow. Then $f\mu _{x}:x\rightarrow a$ is a
positive arrow. Since $\mu _{x}$ is a completion of $x$, there exists an
arrow $g:a\rightarrow y$ such that $gf\mu _{x}=\mu _{x}$. Since $\mu _{x}$
is a completion of $x$, $gf=1_{y}$. This shows that $1_{y}$ is weakly
terminal in $x\downarrow \mathcal{C}$. It is obvious that it is
quasi-initial.
\end{proof}

Suppose now that $\mathcal{C}$ is a monopole, $\mathcal{E}$ is a category,
and $U:\mathcal{E}\rightarrow \underline{\mathcal{C}}$ is a functor. Let $%
x\downarrow _{\mathcal{C}}U$ be the category whose objects are pairs $\left(
b,f\right) $ where $b$ is an object of $\mathcal{E}$, and $f:a\rightarrow Ub$
is a positive arrow in $\mathcal{C}$. A morphism $\left( b,f\right)
\rightarrow \left( c,g\right) $ in $x\downarrow _{\mathcal{C}}U$ is a
morphism $\varphi :b\rightarrow c$ in $\mathcal{E}$ such that $U\varphi
\circ f=g$. Notice that this coincides with the underlying category of $%
x\downarrow \mathcal{C}$ when $U$ is the identity functor of $\underline{%
\mathcal{C}}$.

\begin{definition}
\label{Definition:hull-functor}Let $\mathcal{C}$ and $\mathcal{E}$ be
electricategories, $U:\mathcal{E}\rightarrow \mathcal{C}$ be a functor, and $%
x$ be an object of $\mathcal{C}$.\ A completion or hull $\mathrm{h}_{%
\mathcal{C}}^{U}\left( x\right) $ of $x$ in $\mathcal{C}$ \emph{with respect
to} $U$ is an \emph{terminal object} in $x\downarrow _{\mathcal{C}}U$.
\end{definition}

Notice that in Definition \ref{Definition:hull-functor} the completion is
required a terminal object, rather than an amphi-terminal object. This
discrepancy will be reconciled in Theorem \ref{Theorem:main}, where it is
explained how completions with respect to a functor will produce
\textquotedblleft absolute\textquotedblright\ completions as in Definition %
\ref{Definition:hull}.

\begin{definition}
\label{Definition:relative-hull}Let $\mathcal{C}$ be a monopole, $x$ be an
object of $\mathcal{C}$, and $\mathcal{H}$ a refinement of $\mathcal{C}_{+}$%
. We let $\mathcal{C}|_{\mathcal{H}}$ be the sub-monopole of $\mathcal{C}$
whose positive arrows are precisely the arrows of $\mathcal{H}$. A
completion $\mathrm{h}_{\mathcal{C}}\left( x|\mathcal{H}\right) $ of $x$ in $%
\mathcal{C}$ \emph{relative }to $\mathcal{H}$ is a completion of $x$ in the
monopole $\mathcal{C}|_{\mathcal{H}}$. We also denote the co-slice category $%
x\downarrow (\mathcal{C}|_{\mathcal{H}})$ as $x\downarrow _{\mathcal{H}}%
\mathcal{C}$.
\end{definition}

Likewise, one can define the completion of $x$ in $\mathcal{C}$ with respect
to a functor $U$ \emph{relative to} $\mathcal{H}$.

\begin{remark}
Various notions of \emph{completion }arise naturally in several areas of
mathematics, also under the name of \emph{envelope }or \emph{hull} or \emph{%
closure}. Sometime in their abstract definition the \emph{rigidity }properly
corresponding to the \emph{uniqueness} requirement in the definition of
amphi-terminal object (Definition \ref{Definition:amphi-terminal}) is
omitted. This is due to the fact that, generally, such a rigidity does not
hold in relation to the whole category. However, in most cases such a
rigidity does hold \emph{relatively }to a suitable class of morphisms.
\end{remark}

\subsection{Injectivity}

Let $\mathcal{C}$ be a polarity, and $x$ be an object of $\mathcal{C}$. We
define injective objects in a (positive) monopole in terms of an \emph{%
extension property }relative to positive arrows.

\begin{definition}
\label{Definition:injective-monopole}Let $\mathcal{C}$ be a monopole.

\begin{itemize}
\item An object $x$ of a monopole $\mathcal{C}$ is \emph{injective }in $%
\mathcal{C}$ if for every positive arrow $g:a\rightarrow b$ and every arrow $%
f:a\rightarrow x$ there exists an arrow $h:b\rightarrow x$ such that $hg=f$.

\item The monopole $\mathcal{C}$ has enough injectives if for every object $%
a $ of $\mathcal{C}$ there exists a positive arrow $a\rightarrow x$ in $%
\mathcal{C}$ with injective $x$.
\end{itemize}
\end{definition}

In the terminology of \cite[Definition 2.1]{adamek_injective_2002},
Definition \ref{Definition:injective-monopole} is asserting that $x$ is $%
\mathcal{C}_{+}$-injective in the category $\underline{\mathcal{C}}$.

We consider more generally the notion of injective object of a polarity.

\begin{definition}
\label{Definition:injective-polarity}Let $\mathcal{C}$ be a polarity.

\begin{itemize}
\item An object $x$ of a monopole $\mathcal{C}$ is \emph{injective }in $%
\mathcal{C}$ if for every positive arrow $g:a\rightarrow b$ and every
negative arrow $f:a\rightarrow x$ there exists an arrow $h:b\rightarrow x$
such that $hg=f$.

\item The polarity $\mathcal{C}$ has enough injectives if for every object $%
a $ of $\mathcal{C}$ there exists a positive arrow $a\rightarrow x$ in $%
\mathcal{C}$ with injective $x$.
\end{itemize}
\end{definition}

Definition \ref{Definition:injective-monopole} is the particular instance of
Definition \ref{Definition:injective-polarity} in the case of a positive
monopole (considered as a polarity where all arrows are negative).

The following lemma essentially shows that, in a context where completions
exist, any injective object is complete.

\begin{lemma}
\label{Lemma:injective-is-complete}Let $\mathcal{C}$ be a polarity, and $%
\mathcal{H}\subseteq \mathcal{C}_{+}$ be a refinement. Suppose that $x$ is
an object of $\mathcal{C}$, and let $\eta _{x}:x\rightarrow Ex$ be a
completion of $x$ in $\mathcal{C}$ with respect to $\mathcal{H}$. Then $\eta
_{x}$ is an isomorphism, and $x$ is complete.
\end{lemma}

\begin{proof}
By injectivity there exists an arrow $g:Ex\rightarrow x$ such that $g\eta
_{x}=1_{x}$. Thus, we have that $\eta _{x}g\eta _{x}=\eta _{x}$ and by
rigidity $\eta _{x}g=1$. This shows that $\eta _{x}$ is an isomorphism.
\end{proof}

\subsection{From completions to capacitors}

Recall the notion of\emph{\ arrow from an object to a functor} from \cite[%
Section III.1]{mac_lane_categories_1998}. If $U:\mathcal{E}\rightarrow 
\mathcal{C}$ is a functor between categories, then an arrow from an object $%
x $ of $\mathcal{C}$ to $U$ is simply a pair $\left( \eta ,a\right) $ where $%
a$ is an object of $\mathcal{E}$ and $\eta :x\rightarrow Ua$ is an arrow in $%
\mathcal{C}$.

\begin{definition}
\label{Definition:rigid}Let $\mathcal{C}$ and $\mathcal{E}$ be categories,
and $U:\mathcal{E}\rightarrow \mathcal{C}$ be a functor. For every object $x$
of $\mathcal{C}$ let $\left( \eta _{x},Jx\right) $ be an arrow from $x$ to $%
U $. This means that $Jx$ is an object of $\mathcal{E}$ and $\eta
_{x}:x\rightarrow UJx$ is a morphism. Say that the family $\left( \eta
_{x},Jx\right) _{x\in \mathcal{C}}$ is \emph{rigid }with respect to $U$ if:

\begin{enumerate}
\item for every arrow $f:x\rightarrow y$ in $\mathcal{C}$ there exists at
most one arrow $\psi :Jx\rightarrow Jy$ in $\mathcal{E}$ such that the square%
\begin{equation*}
\begin{array}{ccc}
x & \overset{f}{\rightarrow } & y \\ 
\eta _{x}\downarrow &  & \downarrow \eta _{y} \\ 
UJx & \underset{U\psi }{\rightarrow } & UJy%
\end{array}%
\end{equation*}%
commutes;

\item for every arrow $f:UJx\rightarrow UJx$ in $\mathcal{C}$ for which the
square%
\begin{equation*}
\begin{array}{ccc}
x & \overset{\mathrm{id}_{x}}{\rightarrow } & x \\ 
\eta _{x}\downarrow &  & \downarrow \eta _{x} \\ 
UJx & \underset{f}{\rightarrow } & UJx%
\end{array}%
\end{equation*}%
commutes, we have $f=\mathrm{id}_{Jx}$.
\end{enumerate}
\end{definition}

Let us assume that $\mathcal{C}$ is a monopole, $\mathcal{E}$ is a category,
and $U:\mathcal{E}\rightarrow \underline{\mathcal{C}}$ is a functor. Let
also $\mathcal{H}$ be a refinement of $\mathcal{C}_{+}$. Suppose that every
object $x$ of $\mathcal{C}$ admits a completion $\eta _{x}:x\rightarrow Jx$
in $\mathcal{C}$ \emph{with respect to} $U$ \emph{relative to }$\mathcal{H}$%
. Then the universal property of the completion (Definition \ref%
{Definition:hull-functor}) produces a functor 
\begin{equation*}
J_{+}:\mathcal{C}_{+}^{\mathrm{op}}\rightarrow \mathcal{E}\text{,\quad }%
x\mapsto Jx\text{,\quad }f\mapsto J_{+}f\text{.}
\end{equation*}%
For an arrow $f$ of $\mathcal{C}_{+}$, the arrow $J_{+}f:Jx\rightarrow Jy$
in $\mathcal{E}$ is uniquely determined by the requirement that the diagram%
\begin{equation*}
\begin{array}{ccc}
x & \overset{f}{\rightarrow } & y \\ 
\eta _{x}\downarrow  &  & \downarrow \eta _{y} \\ 
UJx & \underset{UJ_{+}f}{\leftarrow } & UJy%
\end{array}%
\end{equation*}%
commutes. By definition, we have a natural transformation $\eta :\dagger _{%
\mathcal{C}_{+}}\Rightarrow UJ_{+}$ where 
\begin{equation*}
\dagger _{\mathcal{C}_{+}}:\mathcal{C}_{+}^{\mathrm{op}}\rightarrow \mathcal{%
C}_{+}
\end{equation*}%
is the canonical arrow-reversing functor. We will use these observations in
the following:

\begin{definition}
\label{Definition:capacitor}Let $\mathcal{C}$ be a (positive) monopole. A 
\emph{capacitor }for $\mathcal{C}$ is given by:

\begin{itemize}
\item a refinement $\mathcal{H}$ of $\mathcal{C}_{+}$;

\item a category $\mathcal{E}$;

\item a \emph{faithful} functor $U:\mathcal{E}\rightarrow \underline{%
\mathcal{C}}$;
\end{itemize}

such that:

\begin{enumerate}
\item every object $x$ of $\mathcal{C}$ has a completion $\mu
_{x}:x\rightarrow Jx$ in $\mathcal{C}_{+}$ with respect to $U$ relative to $%
\mathcal{H}$, giving rise to a functor 
\begin{equation*}
J_{+}:\mathcal{C}_{+}^{\mathrm{op}}\rightarrow \mathcal{E}
\end{equation*}%
and a natural transformation 
\begin{equation*}
\eta :\dagger _{\mathcal{C}_{+}}\Rightarrow J_{+}\text{;}
\end{equation*}

\item the resulting family $\left( x,Jx\right) _{x\in \mathcal{C}}$ is \emph{%
rigid} with respect to $U:\mathcal{E}\rightarrow \mathcal{C}$.
\end{enumerate}
\end{definition}

The notions of capacitor and voltage should be compared with Kock--Z\"{o}%
berlein (KZ) doctrines and Yoneda structures on $2$-categories \cite%
{walker_yoneda_2018,street_yoneda_1978,marmolejo_kan_2012,di_liberti_accessibility_2023}%
. Also in the context of KZ doctrines one has for each object $a$ an arrow $%
\sigma _{a}:a\rightarrow Sa$. In that case, the \emph{universal property }of
such arrows requires to have, for each arrow $f:a\rightarrow Sb$, a left Kan
extension of $f$ along $\sigma _{a}$ subject to some natural requirements 
\cite[Definition 2.1]{di_liberti_accessibility_2023}. Also in that context,
a second \textquotedblleft contravariant\textquotedblright\ universal
property, that holds for the so-called \emph{admissible }arrows, is also
considered \cite[Definition 2.18]{di_liberti_accessibility_2023}. The
construction of involutive monads starting from certain contravariant
structures (\emph{construction standard contravariant}) is the subject of 
\cite{guitart_monades_1975}.

\subsection{From capacitors to voltages}

Recall that if $\mathcal{E}$ is a \emph{full} subcategory of $\mathcal{C}$,
then one says that $\mathcal{E}$ is a \emph{reflective subcategory }if the
inclusion functor $U:\mathcal{E}\rightarrow \mathcal{C}$ has a left adjoint $%
F$ \cite[Section IV.3]{mac_lane_categories_1998}. One can think of $\mathcal{%
E}$ as a category of \textquotedblleft complete\textquotedblright\ objects,
and $F$ as a \textquotedblleft completion\textquotedblright\ functor. If $%
\eta $ is the unit of the adjunction, then for each object $x$ of $\mathcal{C%
}$, $\eta _{x}:x\rightarrow Fx$ maps $x$ to its \textquotedblleft
completion\textquotedblright\ $Fx$. For example, one can consider the
category $\mathcal{E}$ of complete metric spaces, as a full subcategory of
the category $\mathcal{C}$ of arbitrary metric spaces and nonexpansive maps.
In this case, the functor $F$ maps each metric space to its
(Cauchy-)completion.

We following result can be seen in particular as a generalization of
reflectors to a context where, as it is often the case, they cannot be
defined on the whole category one begins with:

\begin{theorem}
\label{Theorem:main}Let $\mathcal{C}$ be a (positive) \emph{monopole}. Let
also $\left( \mathcal{H},U:\mathcal{E}\rightarrow \underline{\mathcal{C}}%
\right) $ be a \emph{capacitor} for $\mathcal{C}$. Consider:

\begin{itemize}
\item the corresponding functor $J_{+}:\mathcal{C}_{+}^{\mathrm{op}%
}\rightarrow \mathcal{E}$, and

\item the corresponding natural transformation $\eta _{x}:\dagger _{\mathcal{%
C}_{+}}\Rightarrow J_{+}$
\end{itemize}

as in Definition \ref{Definition:capacitor}. Define $\mathcal{C}_{-}$ to be
the collection of arrows $f:x\rightarrow y$ for which there exists a
(necessarily unique) arrow 
\begin{equation*}
J_{-}f:Jx\rightarrow Jy
\end{equation*}%
in $\mathcal{E}$ making the diagram%
\begin{equation*}
\begin{array}{ccc}
x & \overset{f}{\rightarrow } & y \\ 
\eta _{x}\downarrow &  & \downarrow \eta _{y} \\ 
UJx & \underset{UJ_{-}f}{\rightarrow } & UJy%
\end{array}%
\end{equation*}%
commutes. Define $E_{+}:=UJ_{+}$, $E_{-}=EJ_{-}$, and $Ex:=UJx$. Then:

\begin{enumerate}
\item $\mathcal{C}_{-}$ is a refinement of $\mathcal{C}$, turning $\mathcal{C%
}$ into a polarity with $\mathcal{C}_{-}$ as set of negative arrows;

\item $J_{-}:\mathcal{C}_{-}\rightarrow \mathcal{E}$ is a functor;

\item $E_{+}$ and $E_{-}$ define a negative endofunctor $E$ of the polarity $%
\mathcal{C}$;

\item $\mu :1_{\mathcal{C}}\Rightarrow E$ is a natural transformation;

\item $E$ and $\mu $ endow $\mathcal{C}$ with the structure of a voltage;

\item for every object $x$ of $\mathcal{C}$, $\mu _{x}:x\rightarrow Ex$ is
the completion of $x$ in $\mathcal{C}$ relative to $\mathcal{H}$;

\item the components of $\eta $ are \emph{monic} and \emph{negatively epic};

\item $E_{-}$ is a fully faithful functor $\mathcal{C}_{-}\rightarrow 
\mathcal{C}_{-}$;

\item for every arrow $f:x\rightarrow y$ in $\mathcal{C}_{-}$, 
\begin{equation*}
E_{-}f:Ex\rightarrow Ey
\end{equation*}%
is the unique arrow in $\mathcal{C}_{-}$ making the diagram 
\begin{equation*}
\begin{array}{ccc}
x & \overset{f}{\rightarrow } & y \\ 
\eta _{x}\downarrow &  & \downarrow \eta _{y} \\ 
Ex & \underset{E_{-}f}{\rightarrow } & Ey%
\end{array}%
\end{equation*}%
commute;

\item $\mu $ is the unit of an adjunction $U\dashv J_{-}$ between the
functors $J_{-}:\mathcal{C}_{-}\rightarrow \mathcal{E}$ and $U:\mathcal{E}%
\rightarrow \mathcal{C}_{-}$;

\item $U$ establishes an equivalence between $\mathcal{E}$ and a reflective
full subcategory of $\mathcal{C}_{-}$;

\item an object $x$ is injective in the polarity $\mathcal{C}$ if and only
if it is complete with respect to $\mathcal{H}$.
\end{enumerate}
\end{theorem}

\begin{proof}
The proof of (1)--(5) is straightforward. The fact that $\eta
_{Ex}:Ex\rightarrow EEx$ is an isomorphism follows from\ (6), which implies
in particular that $Ex$ is closed in $\mathcal{C}_{+}$ relative to $\mathcal{%
H}$.

(6) By hypothesis $\eta _{x}:x\rightarrow Ex$ is the completion of $x$ in $%
\mathcal{C}_{+}$ with respect to $U$ relative to $\mathcal{H}$. Thus, $\eta
_{x}$ is a terminal object in the category $x\downarrow _{\mathcal{H}}U$.
Recall that the objects of $x\downarrow _{\mathcal{H}}U$ are pairs $\left(
\xi ,a\right) $ where $a$ is an object of $\mathcal{E}$ and $\xi
:x\rightarrow Ua$ is an arrow in $\mathcal{H}$. Let $f:x\rightarrow y$ be an
arrow in $\mathcal{H}$. Consider the object $\left( Jy,\eta _{y}\circ
f\right) $ of $x\downarrow _{\mathcal{H}}U$. Then by the universal property
of $\mu _{x}$ there exists an arrow $\psi :Jy\rightarrow Jx$ making the
diagram%
\begin{equation*}
\begin{array}{ccc}
x & \overset{f}{\rightarrow } & y \\ 
\eta _{x}\downarrow &  & \downarrow \eta _{y} \\ 
Ex & \underset{U\psi }{\leftarrow } & Ey%
\end{array}%
\end{equation*}%
commute. This shows that $\eta _{x}$ is weakly terminal in $x\downarrow _{%
\mathcal{H}}{}\mathcal{C}$. Suppose that $f$ is an endomorphism of $Ex$ in $%
\mathcal{C}$ such that the diagram%
\begin{equation*}
\begin{array}{ccc}
x & \overset{\mathrm{id}_{x}}{\rightarrow } & x \\ 
\eta _{x}\downarrow &  & \downarrow \eta _{y} \\ 
Ex & \underset{f}{\rightarrow } & Ey%
\end{array}%
\end{equation*}%
commutes. Then by the rigidity hypothesis in the definition of
capacitor---Definition \ref{Definition:capacitor}(2)---we must have $%
f=J_{-}\left( \mathrm{id}_{x}\right) =\mathrm{id}_{Jx}$. This shows that $%
\eta _{x}$ is rigid in $x\downarrow _{\mathcal{H}}\mathcal{C}$. This
concludes the proof that $\eta _{x}$ is the completion of $x$ relative to $%
\mathcal{H}$;

(7) The arrow $\eta _{x}$ is monic in $\mathcal{C}$ by (6) and Lemma \ref%
{Lemma:hull-monic}. We now show that $\eta _{x}:x\rightarrow Jx$ is \emph{%
negatively epic }in $\mathcal{C}$; see Definition \ref{Definition:epic}.
Suppose that $g,h:Jx\rightarrow y$ are arrows in $\mathcal{C}_{-}$ such that 
$g\eta _{x}=h\eta _{x}$. By definition of $\mathcal{C}_{-}$ and $J_{-}$, we
have that $J_{-}g:Jx\rightarrow Jy$ is the only morphism in $\mathcal{E}$
such that the diagram%
\begin{equation*}
\begin{array}{ccc}
Jx & \overset{g}{\rightarrow } & y \\ 
1_{Jx}\downarrow &  & \downarrow \eta _{y} \\ 
Jx & \underset{E_{-}g}{\rightarrow } & Jy%
\end{array}%
\end{equation*}%
commutes, i.e., such that $E_{-}g=\eta _{y}g$. Likewise, $%
J_{-}h:Jx\rightarrow Jy$ is the only morphism in $\mathcal{E}$ that makes
the diagram%
\begin{equation*}
\begin{array}{ccc}
x & \overset{g}{\rightarrow } & Jy \\ 
\eta _{x}\downarrow &  & \downarrow 1_{Jy} \\ 
Jx & \underset{E_{-}h}{\rightarrow } & Jy%
\end{array}%
\end{equation*}%
commute, i.e., such that $E_{-}h=\eta _{y}h$. Consider the commuting diagrams%
\begin{equation*}
\begin{array}{ccc}
x & \overset{g\eta _{x}}{\rightarrow } & y \\ 
\eta _{x}\downarrow &  & \downarrow \eta _{y} \\ 
Jy & \underset{\eta _{y}g}{\rightarrow } & Jy%
\end{array}%
\text{\quad and\quad }%
\begin{array}{ccc}
x & \overset{h\eta _{x}}{\rightarrow } & y \\ 
\eta _{x}\downarrow &  & \downarrow \eta _{y} \\ 
Jy & \underset{\eta _{y}h}{\rightarrow } & Jy%
\end{array}%
\end{equation*}%
we conclude that%
\begin{equation*}
\left( E_{-}h\right) \eta _{x}=\left( \eta _{y}h\right) \eta _{x}=\eta
_{y}\left( h\eta _{x}\right) =\eta _{y}\left( g\eta _{x}\right) =\left( \eta
_{y}g\right) \eta _{x}=\left( E_{-}g\right) \eta _{x}\text{.}
\end{equation*}%
By the rigidity of $\eta $ we conclude that 
\begin{equation*}
J_{-}h=J_{-}g
\end{equation*}%
and hence%
\begin{equation*}
\eta _{y}h=E_{-}h=E_{-}g=\eta _{y}g\text{.}
\end{equation*}%
Since $\eta _{y}$ is monic in $\mathcal{C}$ we conclude that $g=h$.

(8) Let $f:x\rightarrow y$ be an arrow. Then $J_{-}f:Jx\rightarrow Jy$
witnesses that $E_{-}f:Ex\rightarrow Ey$ belongs to $\mathcal{C}_{-}$
considering the isomorphism $\eta _{Jx}\cong 1_{Jx}$ and $\eta _{Jy}\cong
1_{Jy}$. The fact that $E_{-}$ is faithful follows from the assumption that $%
U$ be faithful in Definition \ref{Definition:capacitor} together with the
fact established in (5) that the arrows $\eta _{x}$ are monic. Finally, the
fullness of $E_{-}$ is manifest by the definition of $\mathcal{C}_{-}$.

(9) Suppose that $f:x\rightarrow y$ is an arrow in $\mathcal{C}_{-}$. Let $%
g:Ex\rightarrow Ey$ be an arrow such that%
\begin{equation*}
\begin{array}{ccc}
x & \overset{f}{\rightarrow } & y \\ 
\eta _{x}\downarrow &  & \downarrow \eta _{y} \\ 
Ex & \underset{g}{\rightarrow } & Ey%
\end{array}%
\end{equation*}%
commutes. We need to prove $g=E_{-}f$. By (8) we have $g=E_{-}h$ for some
arrow $h:x\rightarrow y$. Applying the \emph{uniqueness }requirement in
Definition \ref{Definition:rigid}(1) to the arrows $J_{-}f$ and $J_{-}h$ of $%
\mathcal{E}$ we obtain $J_{-}f=J_{-}h$ and hence $f=E_{-}f=E_{-}h=g$.

(10) This follows from the definition of $\mathcal{C}_{-}$, the fact that $%
\eta _{J}$ is a natural isomorphism, and the uniqueness of the arrow $J_{-}f$
in the hypotheses.

(11)\ The functor $U:\mathcal{E}\rightarrow \mathcal{C}_{-}$ is faithful by
hypothesis and full by (8). The conclusion thus follows from (10).

(12) We being with showing that every complete object is injective. Suppose
that $x$ is complete. Let $g:a\rightarrow b$ be a positive arrow and $%
f:a\rightarrow b$ is a negative arrow. Then $E_{-}f:Ea\rightarrow Ex\cong x$
is an arrow such that $E_{-}f\circ \mu _{a}=f$. Then $E_{+}g:b\rightarrow Ea$
such that $E_{+}g\circ g=\mu _{a}$. Thus we have that $E_{-}f\circ
E_{+}f:b\rightarrow x$ satisfies 
\begin{equation*}
E_{-}f\circ E_{+}g\circ g=E_{-}f\circ \mu _{a}=f\text{.}
\end{equation*}%
This concludes the proof that $x$ is injective.

Conversely, if an object is injective, then it is complete by Lemma \ref%
{Lemma:injective-is-complete} and (6).
\end{proof}

\begin{corollary}
\label{Corollary:main}Under the hypotheses of Theorem \ref{Theorem:main},
the following assertions are equivalent:

\begin{enumerate}
\item the positive monopole $\mathcal{C}$ has enough injectives;

\item every complete object in $\mathcal{C}$ with respect to $\mathcal{H}$
is injective in the positive monopole $\mathcal{C}$;

\item the injective objects of the positive monopole $\mathcal{C}$ are
precisely the objects that are complete with respect to $\mathcal{C}$.
\end{enumerate}
\end{corollary}

\begin{proof}
(1)$\Rightarrow $(2) Let $a$ be a complete object. Let $g:a\rightarrow z$ be
a positive arrow with target $z$. Then since $a$ is complete we have that $%
E_{+}g:z\rightarrow Ea\cong a$ is an arrow such that $E_{+}g\circ g=1_{a}$.
Suppose now that $\varphi :x\rightarrow y$ is a positive arrow in $\mathcal{C%
}$ and $f:x\rightarrow a$ is an arrow in $\mathcal{C}$.\ Applying the
hypothesis of injectivity of $z$ to $\varphi $ and $g\circ f:x\rightarrow z$
we obtain an arrow $f^{\prime }:y\rightarrow z$ such that $f^{\prime }\circ
\varphi =g\circ f$. Thus, $E_{+}g\circ f^{\prime }\circ \varphi =E_{+}g\circ
g\circ f=f$. This shows that $a$ is injective.

(2)$\Rightarrow $(3) Every complete object in $\mathcal{C}$ with respect to $%
\mathcal{H}$ is injective by assumption.\ The converse holds by Lemma \ref%
{Lemma:injective-is-complete}.

(3)$\Rightarrow $(1) The arrows $\eta _{x}$ for $x$ in $\mathcal{C}$ witness
that the positive monopole $\mathcal{C}$ has enough injectives.
\end{proof}

The conclusions of Theorem \ref{Theorem:main} justify the following:

\begin{definition}
\label{Definition:create-hull} Let $\mathcal{C}$ be a monopole, and $\left( 
\mathcal{H},U:\mathcal{E}\rightarrow \underline{\mathcal{C}}\right) $ be a
capacitor for $\mathcal{C}$. Then the functor $U$ \emph{creates }completions
in $\mathcal{C}$ with respect to $\mathcal{H}$.
\end{definition}

\section{Boolean algebras and their completions\label{Section:Boolean}}

In this section we present some applications of the above framework to
Boolean algebras. Particularly, we show that the canonical forgetful functor
from complete Boolean algebras and \emph{complete} homomorphisms creates
hulls in the category of Boolean algebras. The hull of a Boolean algebra is
precisely its completion. As a reference for Boolean algebras we refer to
the monographs \cite{sikorski_boolean_1969,koppelberg_handbook_1989}.

\subsection{Categories of Boolean algebras}

In this section we let $\mathcal{C}$ be the category of Boolean algebras 
\cite[Section 5.1]{koppelberg_handbook_1989}. Recall that a poset $P$ is:

\begin{itemize}
\item a \emph{lattice} if any two elements $x,y$ of $P$ have a supremum $%
x\vee y$ (or join) and an infimum $x\wedge y$ (or meet);

\item \emph{complete lattice }if any subset of $P$ has a supremum and an
infimum;

\item a \emph{bounded lattice }if it is a lattice with a largest element $1$
and least element $0$;

\item a \emph{distributive lattice }if $\wedge $ is distributive with
respect to $\vee $ and (hence \cite[Lemma 1.17]{koppelberg_handbook_1989})
vice versa;
\end{itemize}

see \cite[Definition 8.1]{schroder_ordered_2016}.

For elements $x,y$ of a bounded lattice, set $x\bot y$ if $x\wedge y=0$. An
element $y$ of $L$ is a \emph{complement }of $x$ if $x\bot y$ and $x\vee y=1$
\cite[Definition 1.19]{koppelberg_handbook_1989}. When it exists, such a
complement is unique \cite[Lemma 1.20]{koppelberg_handbook_1989}. A bounded
lattice is complemented if each element has a complement. A \emph{Boolean
algebra} is a complemented bounded distributive lattice \cite[Section 1.5]%
{koppelberg_handbook_1989}.

The morphisms in $\mathcal{C}$ are the Boolean algebra homomorphisms. Notice
that a function between Boolean algebras is a homomorphism as long it
preserves meets and negation or, dually, join and negation.

It is easily seen that the monics in $\mathcal{C}$ are precisely the
injective homomorphisms, also called \emph{embeddings }\cite[Definition 2.5]%
{koppelberg_handbook_1989}. An embedding between Boolean algebras is
necessarily an isomorphism onto its image, which easily implies that any
monic in $\mathcal{C}$ is \emph{regular}. As noted in \cite%
{banaschewski_strong_2010} the epimorphisms in $\mathcal{C}$ are the \emph{%
surjective }homomorphisms; see also \cite[Lemma 4]%
{banaschewski_categorical_1967}.

Let $B$ be a Boolean algebra, and $A\subseteq B$ be a subalgebra. Then $A$
is \emph{dense }in $B$ if every element of $A$ is the supremum of elements
of $A$ \cite[Definition 4.8]{koppelberg_handbook_1989}; see also \cite[%
Definition 4.9]{koppelberg_handbook_1989}. This is equivalent to the
assertion that $A$ has nontrivial intersection with every nontrivial ideal 
\cite[Definition 5.4]{koppelberg_handbook_1989}. Every dense subalgebra $A$
of $B$ is in particular \emph{regular }\cite[Proposition 4.17]%
{koppelberg_handbook_1989}. This means that the inclusion $A\rightarrow B$
preserve arbitrary joins and meets (when they exist in $A$).

A morphism in $\mathcal{C}$ is an \emph{essential monic }if and only if it
is an embedding with dense image, in which case we say that it is an \emph{%
essential embedding}; see \cite[Lemma 5]{banaschewski_categorical_1967}.
(Recall that in general an essential monic is an arrow $f$ such that for any
arrow $g$, if $gf$ is monic, then $g$ is monic.)

\subsection{Completions}

Every Boolean algebra $A$ can be realized as a dense subalgebra of a
complete Boolean algebra \cite[Theorem 4.19]{koppelberg_handbook_1989}. Such
a complete Boolean algebra is essentially unique and denoted by $M\left(
A\right) $. We let $\eta _{A}:A\rightarrow M\left( A\right) $ be the
canonical embedding.

Let $f:A\rightarrow B$ be an embedding between Boolean algebras. By
Sikorski's Extension Theorem \cite[Theorem 5.9]{koppelberg_handbook_1989},
there exists a homomorphism $g:B\rightarrow M\left( A\right) $ such that $%
gf=\eta _{A}$. The homomorphism $g$ is an embedding if and only if $f$ is an
essential embedding. It follows from this that the injective objects in the
category $\mathcal{C}$ of Boolean algebras are precisely the complete
Boolean algebras \cite[Theorem 5.13]{koppelberg_handbook_1989}; see also 
\cite[Theorem 1]{halmos_injective_1961} and \cite[Proposition 3]%
{banaschewski_categorical_1967}.

\subsection{Continuous homomorphism}

Let $A,B$ be Boolean algebras. Let us say that a homomorphism $%
f:A\rightarrow B$ is\emph{\ continuous} if for any subset $X$ of $A$, 
\begin{equation*}
\bigwedge X=0\Rightarrow \bigwedge f\left( X\right) =0\text{;}
\end{equation*}%
see \cite[Definition 5.1]{koppelberg_handbook_1989}. Notice that any
essential embedding is, in particular, continuous. The proof of the
following lemma is immediate.

\begin{lemma}
Let $A$ and $B$ be Boolean algebras, and $f:A\rightarrow B$ be a
homomorphism. The following assertions are equivalent:

\begin{enumerate}
\item $f$ is continuous;

\item $f$ extends to a continuous homomorphism $M\left( A\right) \rightarrow
M\left( B\right) $.
\end{enumerate}

Furthermore, in this case there exists a unique continuous homomorphism $%
M\left( A\right) \rightarrow M\left( B\right) $ that extends $f$.
\end{lemma}

\subsection{Completions as hulls}

Let also $\mathcal{E}$ be the category of \emph{complete }Boolean algebras.
We define the morphisms in $\mathcal{E}$ to be the \emph{continuous }%
homomorphisms. We have a canonical forgetful functor $U:\mathcal{E}%
\rightarrow \mathcal{C}$. One can also consider the objects of $\mathcal{E}$
as topological Boolean algebras, endowed with the order topology.
Completeness of a Boolean algebra corresponds to the assertion that every
monotone net has a limit, while completeness of a homomorphism corresponds
to continuity.

The following result is a consequence of the above remarks. In particular,
it realizes the injective envelope of a Boolean algebra as a hull.
Furthermore, it proves that such an injective envelope is \emph{created }by
the forgetful functor $U$ in the sense of Definition \ref%
{Definition:create-hull}.

\begin{proposition}
\label{Proposition:husk-for-Boolean-algebras}Let:

\begin{itemize}
\item $\mathcal{C}$ be monopole of Boolean algebras with Boolean algebra
homomorphisms as morphisms and embeddings as positive morphisms;

\item $\mathcal{H}$ be the refinement of $\mathcal{C}$ comprising the \emph{%
essential }embeddings;

\item $\mathcal{E}$ be the category of complete Boolean algebras, whose
morphisms are the continuous homomorphisms;

\item $U:\mathcal{E}\rightarrow \underline{\mathcal{C}}$ be the forgetful
functor.
\end{itemize}

Then:

\begin{enumerate}
\item $\left( \mathcal{H},U:\mathcal{E}\rightarrow \underline{\mathcal{C}}%
\right) $ is a capacitor for $\mathcal{C}$;

\item the corresponding negative arrows in $\mathcal{C}$ are the \emph{%
continuous homomorphisms};

\item the functor $U$ creates hulls of Boolean algebras in the positive
monopole of Boolean algebras relative to the refinement of essential
embeddings;

\item for a Boolean algebra $A$, the corresponding hull is the canonical
inclusion of $A$ into its completion.
\end{enumerate}
\end{proposition}

\section{Ordered spaces and their MacNeille completion\label{Section:posets}}

It this section we considered ordered sets (posets), and recognize the
McNeille completion of an ordered set as a hull created by a suitable
forgetful functor. For a general reference on ordered sets, one can consult 
\cite{schroder_ordered_2016,rudeanu_sets_2012,harzheim_ordered_2005}.

\subsection{Category of posets}

In this section we let $\mathcal{C}$ be the category of ordered sets
(posets) \cite{banaschewski_categorical_1967}. The morphisms in $\mathcal{C}$
are the monotone (i.e., order-preserving) maps.

In the category of posets, the monics are the injective monotone maps, while
the epics are the surjective monotone maps \cite[Lemma 1]%
{banaschewski_categorical_1967}. An \emph{embedding }between posets is an
order-isomorphism onto its image. These are precisely the \emph{regular}
monics in $\mathcal{C}$ \cite[Lemma 1]{banaschewski_categorical_1967}.
Unlike the case of Boolean algebras, in the category of posets not every
monic is regular.

Let $B$ be a poset and $A\subseteq B$. Then $A$ is:

\begin{itemize}
\item \emph{join dense }if every element of $B$ is the supremum of $\left\{
a\in A:a\leq b\right\} $;

\item \emph{meet dense }if every element of $B$ is the infimum of $\left\{
a\in A:b\leq a\right\} $;

\item \emph{dense }if it is both meet and join dense.
\end{itemize}

Same definitions apply in the case of an embedding, by considering its image.

Recall that an \emph{essential regular monic} is a regular monic $f$ such
that for any other arrow $g$, if $gf$ is a regular monic, then $g$ is a
regular monic. In the category of posets, the essential regular monics are
precisely the \emph{dense embeddings }\cite[Lemma 3]%
{banaschewski_categorical_1967}.

\subsection{Completions}

The \emph{MacNeille completion} of a poset---also known as \emph{normal
completion }or Dedekind--MacNeille completion---generalizes Dedekind's
construction of the real numbers from the rationals \cite%
{macneille_partially_1937}; see also \cite[Chapter 7]%
{davey_introduction_1990}, \cite[Section 8.3]{schroder_ordered_2016}, and 
\cite[Section 1.10]{harzheim_ordered_2005}.

Let $P$ be a poset. Given a subset $A$ of $P$ define $\uparrow A$ to be the
set of upper bounds of $A$, and $\downarrow A$ the set of lower bounded of $%
A $ \cite[Definition 3.22]{schroder_ordered_2016}. The MacNeille completion
of $P$ is the set%
\begin{equation*}
M\left( P\right) :=\left\{ A\subseteq P:A={}\downarrow \uparrow A\right\} 
\text{.}
\end{equation*}%
Notice that%
\begin{equation*}
\uparrow \downarrow A={}\uparrow \left( \bigvee A\right)
\end{equation*}%
by \cite[Lemma 3.23]{schroder_ordered_2016}. Then $M\left( P\right) $ is a
complete lattice \cite[Definition 8.21]{schroder_ordered_2016}. One can
identify $p\in P$ with the element $\downarrow p\in M\left( P\right) $. This
produces a dense embedding $P\rightarrow M\left( P\right) $ \cite[Theorem
8.23]{schroder_ordered_2016}. We identify $P$ with its image in $M\left(
P\right) $. In the case when $A$ is a Boolean algebra regarded as an ordered
set, its Boolean algebra completion is isomorphic to its MacNeille
completion \cite[Section 5]{banaschewski_categorical_1967}; see also \cite[%
Theorem 6.13]{blyth_lattices_2005}.

The following is an immediate consequence of \cite[Theorem 1.10]%
{harzheim_ordered_2005}---see also \cite[Proposition 2]%
{banaschewski_categorical_1967}---asserting the essential uniqueness of the
MacNeille completion:

\begin{lemma}
\label{Lemma:poset-embedding}Let $f:P\rightarrow Q$ be a dense monotone
embedding between posets. Then there exists a monotone map $g:Q\rightarrow
M\left( P\right) $ such that $gf:P\rightarrow M\left( P\right) $ is the
canonical inclusion.
\end{lemma}

\subsection{Continuous maps}

Let $f:P\rightarrow Q$ be a monotone map between ordered sets. Let us say
that $f$ is \emph{continuous }if for every subset $A$ of $P$ and $a\in P$,%
\begin{equation*}
\bigwedge A=a\Rightarrow \bigwedge f\left( A\right) =f\left( a\right)
\end{equation*}%
and%
\begin{equation*}
\bigvee A=a\Rightarrow \bigvee f\left( A\right) =f\left( a\right) \text{.}
\end{equation*}%
As in the case of Boolean algebras, one can characterize continuous maps as
follows.

\begin{lemma}
Let $P,Q$ be posets, and $f:P\rightarrow Q$ be a monotone map. The following
assertions are equivalent:

\begin{enumerate}
\item $f$ is continuous;

\item $f$ extends to a continuous monotone map $M\left( P\right) \rightarrow
M\left( Q\right) $.
\end{enumerate}

Furthermore, in this case the extension of $f$ to a continuous monotone map $%
M\left( P\right) \rightarrow M\left( Q\right) $ is unique.
\end{lemma}

\subsection{MacNeille completion as hull}

We let $\mathcal{E}$ be the category of \emph{complete lattices} with \emph{%
continuous }monotone maps as morphisms. As in the case of Boolean algebras,
we have a canonical forgetful functor $U:\mathcal{E}\rightarrow \mathcal{C}$%
. As a consequence of the above remarks, we can recognize the MacNeille
completion of a poset as a hull.

\begin{proposition}
Let:

\begin{itemize}
\item $\mathcal{C}$ be the monopole of posets with monotone maps as
morphisms and embeddings as positive morphisms;

\item $\mathcal{H}$ be the refinement of $\mathcal{C}$ comprising the \emph{%
essential }embeddings;

\item $\mathcal{E}$ be the category of complete lattices, whose morphisms
are the continuous monotone maps;

\item $U:\mathcal{E}\rightarrow \underline{\mathcal{C}}$ be the forgetful
functor.
\end{itemize}

Then:

\begin{enumerate}
\item $\left( \mathcal{H},U:\mathcal{E}\rightarrow \underline{\mathcal{C}}%
\right) $ is a capacitor for $\mathcal{C}$;

\item the corresponding negative arrows in $\mathcal{C}$ are the \emph{%
continuous monotone maps};

\item the functor $U$ creates hulls of posets in the monopole of posets
relative to the refinement of essential embeddings;

\item for a poset $P$, the corresponding hull is the canonical inclusion of $%
P$ into its MacNeille completion.
\end{enumerate}
\end{proposition}

\section{Rings and their multiplier rings\label{Section:rings}}

In this section we consider \emph{multiplier rings }of rings, and recognize
them as yet another instance of hulls created by a forgetful functor.

\subsection{Multiplier rings}

Let $\mathcal{R}$ be the variety of rings, which is a semi-abelian category
of algebras \cite[Example 2.6]{janelidze_semi-abelian_2002}; see also \cite%
{janelidze_ideals_2007}. For a ring $R$, define its \emph{annihilator}%
\begin{equation*}
a_{R}\left( R\right) :=\left\{ x\in R:xR=Rx=0\right\} \text{.}
\end{equation*}%
The ring $R$ is\emph{\ non-degenerate} if $a_{R}\left( R\right) $ is trivial 
\cite{van_daele_multiplier_2025}. For such a ring, the (unital) \emph{%
multiplier ring }$M\left( R\right) $ of $R$ has as elements the pairs of
endomorphisms $\left( \lambda ,\rho \right) $ of $A$ satisfying%
\begin{equation*}
\lambda \left( a\right) b=a\rho \left( b\right)
\end{equation*}%
such that $\lambda $ is a left $R$-module homomorphism and $\rho $ is a
right $R$-module homomorphism. The sum is induced by sum of endomorphisms,
and product is given by composition. An element $\xi $ of $R$ yields an
element $\left( \lambda _{\xi },\rho _{\xi }\right) $ on $M\left( R\right) $
defined by%
\begin{equation*}
\lambda _{\xi }:a\mapsto a\xi \text{ and }\rho _{\xi }:a\mapsto \xi a\text{.}
\end{equation*}%
When $R$ is a non-degnerate ring, the assignment $\xi \mapsto \left( \lambda
_{\xi },\rho _{\xi }\right) $ defines an injective homomorphism. By
identifying $R$ with its image inside $M\left( R\right) $, we can regard $R$
as a subring, which is in fact an \emph{essential ideal}, of $M\left(
R\right) $.

We denote an element $\left( \lambda ,\rho \right) $ of $M\left( R\right) $
as $m$ and write, for $a,b\in R$,%
\begin{equation*}
\lambda \left( a\right) =ma\text{ and }\rho \left( b\right) =mb\text{;}
\end{equation*}%
see \cite[Notation 1.3]{van_daele_multiplier_2025}. The \emph{strict
topology }on $M\left( R\right) $ is the weakest topology makes the
evaluation maps 
\begin{equation*}
M\left( R\right) \rightarrow R\text{, }\left( \lambda ,\rho \right) \mapsto
\lambda \left( a\right) \text{\quad and\quad }\left( \lambda ,\rho \right)
\mapsto \rho \left( a\right)
\end{equation*}%
\emph{continuous} for every $a\in R$, where $R$ is endowed with the \emph{%
discrete topology}; see \cite[Definition 2.7]{van_daele_multiplier_2025}.

When $R$ is an algebra over a field $k$, then $M\left( R\right) $ is also a $%
k$-algebra with respect to the scalat multiplication defined by the
identities%
\begin{equation*}
a\left( \lambda m\right) =\lambda \left( am\right) \text{\quad and\quad }%
\left( \lambda m\right) b=\lambda \left( mb\right)
\end{equation*}%
for $m\in M\left( R\right) $ and $a,b\in R$ and $\lambda \in k$.

\subsection{Local units}

For subrings $A,B$ of a ring $C$ define $AB$ be the submodule of $C$
generated by $ab$ for $a\in A$ and $b\in B$. A ring $A$:

\begin{itemize}
\item is \emph{idempotent} if $A$ is equal to $AA$;

\item is \emph{firm }if multiplication induces an isomorphism $A\otimes
_{A}A\rightarrow A$;

\item \emph{has local units }if for any $a\in A$ there exist $e,f\in A$ such
that $ae=fa=a$ \cite[Definition 2.1]{van_daele_multiplier_2025}, which
implies that $A$ is non-degenerate, idempotent, and firm \cite[Lemma 2.4]%
{van_daele_multiplier_2025}.
\end{itemize}

Clearly, any unital ring has local units. If $A$ is a ring with local units,
then for every finite subset $F$ of $A$ there exists $e\in A$ such that $%
ea=ae=a$ for every $a\in F$ \cite[Definition 2.1]{van_daele_multiplier_2025}%
. Furthermore, $A$ is dense in $M\left( A\right) $ endowed with the strict
topology \cite[Proposition 2.9]{van_daele_multiplier_2025}. In particular,
there exists a net $\left( a_{i}\right) $ in $A$, which we call \emph{%
approximate unit }for $A$, strictly converging to $1$ in $M\left( A\right) $.

\subsection{Hopf algebras}

\emph{Multiplier Hopf algebras} have been introduced in \cite%
{van_daele_multiplier_1994} as an algebraic model for locally compact
quantum groups; see also \cite[Chapter 2]{timmermann_invitation_2008}. In
the \emph{unital }case, this recovers the usual notion of \emph{Hopf algebra}%
, which can be seen as an algebraic model for \emph{compact }quantum groups;
see \cite[Chapter 1]{timmermann_invitation_2008}.

As shown in \cite[Proposition 3.3]{van_daele_multiplier_2025}, any
multiplier Hopf algebra $A$ is a \emph{ring with local units}. In
particular, it is non-degenerate, idempotent, firm, and strictly dense into
its multiplier algebra $M\left( A\right) $, which is a Hopf algebra.

\subsection{Non-degenerate homomorphism}

A ring homomorphism $f:A\rightarrow B$ between rings with local units is%
\emph{\ non-degenerate} if $f\left( A\right) B=Bf\left( A\right) =B$ \cite[%
Definition 1.12]{van_daele_multiplier_2025}.

\begin{lemma}
\label{Lemma:non-degenerate}Let $A,B$ rings with local units, and $%
f:A\rightarrow B$ be a ring homomorphism. The following assertions are
equivalent:

\begin{enumerate}
\item $f$ is non-degenerate;

\item $f$ extends to a strictly-continuous unital ring homomorphism $\varphi
:M\left( A\right) \rightarrow M\left( B\right) $.
\end{enumerate}

Furthermore, in this case there exist a unique ring homomorphism $M\left(
A\right) \rightarrow M\left( B\right) $ that extends $f$.
\end{lemma}

\begin{proof}
(1)$\Rightarrow $(2) The unital ring homomorphism $\varphi $ is defined in 
\cite[Definition 1.12]{van_daele_multiplier_2025} by setting%
\begin{equation*}
\varphi \left( m\right) b=\sum_{i}f\left( ma_{i}\right) b_{i}\text{\quad
and\quad }b\varphi \left( m\right) =\sum_{j}b_{j}^{\prime }f(a_{j}^{\prime
}m)
\end{equation*}%
for 
\begin{equation*}
b=\sum_{i}f\left( a_{i}\right) b_{i}=\sum_{j}b_{j}^{\prime }f(a_{j}^{\prime
})
\end{equation*}%
with $a_{i},a_{j}^{\prime }\in A$ and $b_{i},b_{j}^{\prime }\in B$. As
observed there, $\varphi $ is the unique ring homomorphism that extends $f$.
It is obvious from its definition that $\varphi $ is strictly continuous.

(2)$\Rightarrow $(1) Let $\left( a_{i}\right) $ be an approximate unit for $%
A $. Then for $b\in B$ we have that%
\begin{equation*}
b=b1=b\varphi \left( 1\right) =b{}\varphi (\mathrm{lim}_{i}a_{i})=b{}\mathrm{%
lim}_{i}f\left( a_{i}\right) =\mathrm{lim}_{i}{}bf\left( a_{i}\right)
\end{equation*}%
This shows that $b\in Bf\left( A\right) $. The same proof shows $b\in
f\left( A\right) B$.
\end{proof}

\subsection{A capacitor for rings with local units}

Let $\mathcal{C}$ be the category of rings with local units and ring
homomorphisms.\ We regard $\mathcal{C}$ as a polarity, where the positive
arrows are injective ring homomorphisms with an ideal as image (normal
monics), and $\mathcal{C}_{-}$ comprises the non-degenerate ring
homomorphisms. A normal essential monic is an arrow whose image is an
essential ideal. (An ideal is \emph{essential }if it has nontrivial
intersection with any nontrivial ideal.)

Define $\mathcal{E}$ to be the category that has objects \emph{unital
topological rings}, and as morphisms the \emph{continuous unital }ring
homomorphisms. Define $U:\mathcal{E}\rightarrow \mathcal{A}$ be the
forgetful functor that forgets the topology. For a ring with local units $A$%
, define $M\left( A\right) $ to be its multiplier ring endowed with the
strict topology, and $\mu _{A}:A\rightarrow UM\left( A\right) $ be the
canonical inclusion. Then we have that $E\left( A\right) =UM\left( A\right) $
is the multiplier ring of $A$ regarded as a ring without topology.

\begin{lemma}
\label{Lemma:normal-monic}Suppose that $A,B$ are rings with local units, and 
$f:A\rightarrow B$ is an injective homomorphism whose image is an ideal of $%
B $. Then:

\begin{enumerate}
\item there exists a unique unital homomorphism $\psi :M\left( A\right)
\rightarrow M\left( A\right) $ such that $\psi \mu _{B}f=\mu _{A}$;

\item the unital homomorphism $\psi $ in (1) is strictly continuous;

\item the unital homomorphism $\psi $ in (1) is injective if and only if the
image of $f$ is an essential ideal of $B$.
\end{enumerate}
\end{lemma}

\begin{proof}
This is proved in \cite[Proposition 1.7]{van_daele_multiplier_2025}. It is
transparent that for $m\in M\left( B\right) $ and $a\in A$, we must have 
\begin{equation*}
a\cdot \psi \left( m\right) =f\left( a\right) \cdot m\text{ and }\psi \left(
m\right) \cdot a=m\cdot f\left( a\right) \text{.}
\end{equation*}%
This observation yields both the uniqueness of $\psi $ and its strict
continuity.
\end{proof}

\begin{proposition}
\label{Proposition:husk-for-rings}Let:

\begin{itemize}
\item $\mathcal{C}$ be the monopole of rings with local units with ring
homomorphisms as morphisms and injective homomorphisms with an ideal as
image as positive morphisms;

\item $\mathcal{H}$ be the refinement of $\mathcal{C}_{+}$ of injective
homomorphisms with an \emph{essential} ideal as image;

\item $\mathcal{E}$ be the category of topological unital rings, whose
morphisms are the continuous ring homomorphisms;

\item $U:\mathcal{E}\rightarrow \underline{\mathcal{C}}$ be the forgetful
functor.
\end{itemize}

Then:

\begin{enumerate}
\item $\left( \mathcal{H},U:\mathcal{E}\rightarrow \underline{\mathcal{C}}%
\right) $ is a capacitor for $\mathcal{C}$;

\item the corresponding negative arrows in $\mathcal{C}$ are the \emph{%
non-degenerate ring homomorphisms};

\item the functor $U$ creates hulls of rings with local units in the
positive monopole of rings with local units with respect to injective
homomorphisms with essential ideal as image;

\item for a ring with local units $A$, the corresponding hull is the
canonical inclusion of $A$ into its multiplier ring.
\end{enumerate}
\end{proposition}

\section{C*-algebras and their multiplier algebras\label{Section:algebras}}

In this section we recall some results from the theory of C*-algebras, as
can be found in the monographs \cite%
{blecher_operator_2004,blackadar_operator_2006,lance_hilbert_1995}. We then
realize multiplier algebras of C*-algebras as completions created by a
functor.

\subsection{C*-algebras}

Let $\mathcal{A}$ be the category of C*-algebras and *-homomorphisms \cite[%
Chapter II]{blackadar_operator_2006}. Then $\mathcal{A}$ is a semi-abelian
category \cite[Example 2.4]{gran_semi-abelian_2004}. If $A$ is a C*-algebra,
then a closed ideal $I\subseteq A$ is \emph{essential }if it has nontrivial
intersection with any nontrivial ideal \cite[II.5.4.7]%
{blackadar_operator_2006}. A \emph{normal monic} in $\mathcal{A}$ is an
injective *-homomorphism whose image is an ideal, which is necessarily \emph{%
closed} and \emph{self-adjoint}. A \emph{normal} \emph{essential monic }in $%
\mathcal{A}$ is an injective *-homomorphism whose image is an essential
ideal. A *-homomorphism $\varphi :A\rightarrow B$ between C*-algebras is 
\emph{non-degenerate }if $\varphi \left( A\right) B$ is \emph{dense }in $B$ 
\cite[II.7.3.9]{blackadar_operator_2006}, in which case in fact $\varphi
\left( A\right) B=B$ \cite[II.5.3.7]{blackadar_operator_2006} by Cohen's
Factorization Theorem \cite[A.6.1]{blecher_operator_2004}. Thus, this notion
agrees with the notion of non-degenerate ring homomorphism in the algebraic
context.

\subsection{Multiplier algebras}

Let $A$ be a C*-algebra, and $M\left( A\right) $ be its multiplier algebra.
The involution on $A$ induces an involution on $M\left( A\right) $ by the
formulae%
\begin{equation*}
am^{\ast }:=\left( ma^{\ast }\right) ^{\ast }\text{\quad and\quad }m^{\ast
}b:=\left( b^{\ast }m\right) ^{\ast }\text{.}
\end{equation*}%
Furthermore, for $m\in M\left( A\right) $ the linear operators%
\begin{equation*}
A\rightarrow A\text{, }a\mapsto ma\text{\quad and\quad }a\mapsto am
\end{equation*}%
are bounded of the same norm, which is by definition the norm of $m$. This
renders $M\left( A\right) $ a C*-algebra, in such a way that the inclusion $%
\mu _{A}:A\rightarrow M\left( A\right) $ is a non-degenerate *-homomorphism;
see \cite[II.7.3]{blackadar_operator_2006}.

The \emph{strict topology }on $M\left( A\right) $ is the weakest topology on 
$M\left( A\right) $ that makes the functions $M\left( A\right) \rightarrow A$%
, $m\mapsto am$ and $m\mapsto ma$ continuous with respect to the norm
topology on $A$ \cite[II.7.3.11]{blackadar_operator_2006}. This renders the
C*-algebra operations on $M\left( A\right) $ continuous \emph{when
restricted to the unit ball}, which is the subset of $M\left( A\right) $
comprising the elements of norm at most $1$.

One can describe $M\left( A\right) $ in terms of the theory of Hilbert
C*-modules \cite[Chapter 1]{lance_hilbert_1995}. Thus, $A$ can be seen as a
Hilbert C*-module over itself. The C*-algebra $\mathcal{L}\left( A\right) $
of adjointable operators on $A$ is isomorphic to $M\left( A\right) $, via an
isomorphism that maps the algebra $\mathcal{K}\left( A\right) $ of compact
operators onto $A$ \cite[Corollary II.7.3.10]{blackadar_operator_2006}.

The following characterization of non-degenerate *-homomorphisms is the
particular instance of \cite[Proposition 2.5]{lance_hilbert_1995} when the
Hilbert $B$-module $\mathcal{E}$ is equal to $B$, in which case $\mathcal{L}%
\left( \mathcal{E}\right) $ coincides with $M\left( B\right) $.

\begin{lemma}
\label{Lemma:non-degenerate-C*}Let $A,B$ C*-algebras, and $f:A\rightarrow B$
be a *-homomorphism. The following assertions are equivalent:

\begin{enumerate}
\item $f$ is non-degenerate;

\item $f$ extends to a unital *-homomorphism $M\left( A\right) \rightarrow
M\left( B\right) $ that is \emph{strictly continuous }on the unit ball.
\end{enumerate}

Furthermore, in this case there exist a unique *-homomorphism $M\left(
A\right) \rightarrow M\left( B\right) $ that extends $f$.
\end{lemma}

\begin{proof}
The equivalence of (1) and (2) is established in \cite[Proposition 2.5]%
{lance_hilbert_1995}. The uniqueness assertion is established as in the
algebraic case of rings with local units.
\end{proof}

\subsection{A capacitor for C*-algebras}

The following is the analogue of Lemma \ref{Lemma:normal-monic} in the case
of C*-algebras.

\begin{lemma}
\label{Lemma:normal-monic-C*}Suppose that $A,B$ are C*-algebras, and $%
f:A\rightarrow B$ is an injective homomorphism whose image is an ideal of $B$%
. Then:

\begin{enumerate}
\item there exists a unique unital *-homomorphism $\psi :M\left( A\right)
\rightarrow M\left( A\right) $ such that $\psi \mu _{B}f=\mu _{A}$;

\item the unital homomorphism $\psi $ in (1) is strictly continuous on the
unit ball of $M\left( B\right) $;

\item the unital homomorphism $\psi $ in (1) is injective if and only if the
image of $f$ is an essential ideal of $B$.
\end{enumerate}
\end{lemma}

\begin{proof}
This is a particular instance of \cite[Theorem II.7.3.9]%
{blackadar_operator_2006} in the case when the Hilbert $B$-module $\mathcal{E%
}$ is $B$ itself. As in the algebraic case of rings with local units, for $%
m\in M\left( B\right) $ and $a\in A$, we must have 
\begin{equation*}
a\cdot \psi \left( m\right) =f\left( a\right) \cdot m\text{\quad and\quad }%
\psi \left( m\right) \cdot a=m\cdot f\left( a\right) \text{.}
\end{equation*}%
This shows that $\psi $ is strictly continuous on the unit ball of $M\left(
B\right) $.
\end{proof}

Let $\mathcal{E}$ be the category that has as objects \emph{strict unital
C*-algebras}, i.e., \emph{unital }C*-algebras endowed with a topology (\emph{%
strict topology}), generally different from the norm topology, that makes
the C*-algebra operations continuous \emph{when restricted to the unit ball}%
. Then for any C*-algebra $A$ we can regard $M\left( A\right) $ as an object
of $\mathcal{E}$. A morphism between objects of $\mathcal{E}$ is a \emph{%
unital} *-homomorphism that is \emph{strict}, i.e., continuous with respect
to the strict topology when restricted to the unit ball. We have a canonical
forgetful functor $U:\mathcal{E}\rightarrow \mathcal{A}$ that forgets the
topology on a given strict unital C*-algebra.

Combining Lemma \ref{Lemma:non-degenerate-C*} and Lemma \ref%
{Lemma:normal-monic-C*} one obtains the following:

\begin{proposition}
Let:

\begin{itemize}
\item $\mathcal{C}$ be monopole of C*-algebras with *-homomorphisms as
morphisms and injective *-homomorphisms with an ideal as image as positive
morphisms;

\item $\mathcal{H}$ be the refinement of $\mathcal{C}$ comprising the
injective *-homomorphisms with an \emph{essential} ideal as image;

\item $\mathcal{E}$ be the category of unital strict C*-algebras and unital
strict *-homomorphisms;

\item $U:\mathcal{E}\rightarrow \underline{\mathcal{C}}$ be the forgetful
functor.
\end{itemize}

Then:

\begin{enumerate}
\item $\left( \mathcal{H},U:\mathcal{E}\rightarrow \underline{\mathcal{C}}%
\right) $ is a capacitor for $\mathcal{C}$;

\item the corresponding negative arrows in $\mathcal{C}$ are the \emph{%
non-degenerate *-homomorphisms};

\item the functor $U$ creates hulls in the positive monopole of C*-algebras
with respect to essential injective *-homomorphisms with an essential ideal
as image;

\item for a C*-algebra $A$, the corresponding hull is the canonical
inclusion of $A$ into its multiplier algebra.
\end{enumerate}
\end{proposition}

\section{The enveloping von\ Neumann algebra of a C*-algebra\label%
{Section:vN-envelope}}

We conclude in this section by considering the \emph{von Neumann envelope }%
of a C*-algebra, induced by the canonical forgetful functor from von Neumann
algebras to C*-algebras.

\subsection{von\ Neumann algebras}

C*-algebras can be defined, concretely, as norm-closed *-subalgebras of the
algebra $B(H)$ of bounded linear operators on a Hilbert space. Likewise,
von\ Neumann algebras are the *-subalgebras of $B(H)$ that are closed with
respect to the strong (or, equivalently, weak) operator topology \cite[%
Section 2.2.6]{pedersen_algebras_1979}.

\emph{Normal non-degenerate }*-homomorphisms provide the natural notion of
morphism between von Neumann algebras. These are precisely the
*-homomorphisms that preserve the limits of bounded increasing nets of
self-adjoint elements \cite[Section 2.5.1]{pedersen_algebras_1979}. Let $M$
be a von Neumann algebra. Define $M_{\ast }$ to be the Banach space of
normal linear functionals on $M$.\ Then the evaluation map establishes an
isomorphism between $M$ and the dual space of $M_{\ast }$ \cite[Section 3.6.5%
]{pedersen_algebras_1979}, and this characterizes von Neumann algebras among
C*-algebras. The topology on $M$ induced by the w*-topology on the dual of $%
M_{\ast }$ under this isomorphism is called $\sigma $-weak topology. A
*-homomorphism between von Neumann algebras is normal if and only if it is $%
\sigma $-weakly continuous.

\subsection{The von Neumann envelope}

Let $A$ be a C*-algebra. A representation of $A$ is a *-homomorphism $\pi
:A\rightarrow B\left( H\right) $ for some Hilbert space $H$. Such a
representation is non-degenerate if it is so as a *-homomorphism $%
A\rightarrow B\left( H\right) $. By the Bicommutant Theorem \cite[I.9.1.1]%
{blackadar_operator_2006}, this is equivalent to the assertion that the
closure of $\pi \left( A\right) $ in $B\left( H\right) $ with respect to the
weak operator topology is equal to the double commutant $\pi \left( A\right)
^{\prime \prime }$.

Every C*-algebra admits a so-colled \emph{universal representation }$\omega
_{A}:A\rightarrow B\left( H_{A}\right) $ \cite[III.5.2.1]%
{blackadar_operator_2006}. By definition, this is the sum of all the
irreducible representations. The \emph{von Neumann envelope }or \emph{%
enveloping von Neumann algebra }$\mathrm{vN}\left( A\right) $ of $A$ is the
weak closure $\omega _{A}\left( A\right) ^{\prime \prime }$ in $B\left(
H_{A}\right) $ of the image of $A$ under the universal representation \cite[%
III.5.2.5]{blackadar_operator_2006}.

The map $\mathrm{vN}\left( A\right) _{\ast }\rightarrow A^{\ast }$, $\phi
\mapsto \phi \circ \omega _{A}$ defines an isometric isomorphism from the
predual of the von Neumann envelope of $A$ and the dual $A^{\ast }$ of $A$ 
\cite[III.5.2.6]{blackadar_operator_2006}. The corresponding dual map is
thus an isometric isomorphism $A^{\ast \ast }\rightarrow \mathrm{vN}\left(
A\right) $. Thus, $\mathrm{vN}\left( A\right) $ can be identified as a
Banach space with the second dual of $A$. Furthermore, under this
identification the $\sigma $-weak topology on \textrm{vN}$\left( A\right) $
corresponds to the w*-topology on $A^{\ast \ast }$. As a consequence the von
Neumann envelope of $A$ is sometimes just called the second dual of $A$ and
denoted by $A^{\ast \ast }$. Furthermore, the C*-algebra $A$ is identified
with its canonical image inside $A^{\ast \ast }\cong \mathrm{vN}\left(
A\right) $.

Any *-homomorphism $\phi :A\rightarrow B$ admits a unique extension $\phi
^{\ast \ast }:A^{\ast \ast }\rightarrow B^{\ast \ast }$ to a normal
*-homomorphism between their von Neumann envelopes.

When $M$ is a von Neumann algebra, $M$ admits a universal \emph{normal }%
representation $\nu _{M}:M\rightarrow B\left( H_{M}\right) $. This is a
faithful normal representation, and a \emph{direct summand} of the universal
representation. The map $\nu _{M}$ extends to a normal *-homomorphism $%
M^{\ast \ast }\rightarrow \nu _{M}\left( M\right) $ that is the identity on $%
\nu _{M}\left( M\right) $ (conditional expectation).

\subsection{Universal property}

The von Neumann envelope satisfies the following universal property:
whenever $\pi $ is a non-degenerate representation of $A$ on a Hilbert space 
$A$, there exists a unique normal *-homomorphism $\pi ^{\ast \ast }:A^{\ast
\ast }\rightarrow \pi \left( A\right) ^{\prime \prime }$ such that $\pi
^{\ast \ast }|_{A}=\pi $ \cite[Section 2.5.1]{pedersen_algebras_1979}. By
the above remarks, this can be equivalently reformulated as follows: given
any non-degenerate *-homomorphism $\psi :A\rightarrow M$ from $A$ to a von\
Neumann algebra $M$ such that $\psi \left( A\right) $ is $\sigma $-weakly
dense in $M$, there exists a unique normal *-homomorphism $\psi ^{\ast \ast
}:A^{\ast \ast }\rightarrow M$ such that $\psi ^{\ast \ast }|_{A}=A$.

Recall that a normal monic in the category of C*-algebras is an injective
*-homomorphism with an ideal as image. It is furthermore a normal essential
monic if and only if it is an injective *-homomorphism whose image is an
essential ideal.

\begin{definition}
Let $\phi :A\rightarrow B$ be a *-homomorphism between C*-algebras. Then $%
\phi $ is vN-dense if $\phi ^{\ast \ast }:A^{\ast \ast }\rightarrow B^{\ast
\ast }$ has $\sigma $-weakly dense image.
\end{definition}

\begin{lemma}
Let $\phi :A\rightarrow M$ be a *-homomorphism between C*-algebras. Suppose
that $M$ is a von Neumann algebra. The following assertions are equivalent:

\begin{enumerate}
\item $\phi $ is vN-dense;

\item $\phi $ has $\sigma $-weakly dense image.
\end{enumerate}
\end{lemma}

\begin{proof}
(1)$\Rightarrow $(2) By assumption $\phi :A^{\ast \ast }\rightarrow M^{\ast
\ast }$ has $\sigma $-weakly dense image. Identify $M$ with its image under
the universal normal representation and consider the normal conditional
expectation $E_{M}:M^{\ast \ast }\rightarrow M$. Then $E_{M}\circ \phi
^{\ast \ast }$ has $\sigma $-weakly dense image. Since $\left( E_{M}\circ
\phi ^{\ast \ast }\right) |_{A}=\phi $, the same holds for $\phi $;

(2) By hypothesis $\phi \left( A\right) $ is $\sigma $-weakly dense in $M$,
which is $\sigma $-weakly dense in $M^{\ast \ast }$. Thus, $\phi \left(
A\right) \subseteq \phi ^{\ast \ast }\left( A^{\ast \ast }\right) $ is $%
\sigma $-weakly dense in $M^{\ast \ast }$.
\end{proof}

In view of the above remarks, we can see the von Neumann envelope of a
C*-algebra as a hull \emph{created }by the forgetful functor $U$ from von
Neumann algebras to C*-algebras \emph{with respect to }the refinement of 
\emph{vN-dense }injective *-homomorphisms.

\begin{proposition}
Let:

\begin{itemize}
\item $\mathcal{C}$ be monopole of C*-algebras with *-homomorphisms as
morphisms and injective *-homomorphisms as positive morphisms;

\item $\mathcal{H}$ be the refinement of $\mathcal{C}$ comprising the
vN-dense injective *-homomorphisms;

\item $\mathcal{E}$ be the category of von Neumann algebras with $\sigma $%
-weakly continuous *-homomorphisms as morphisms;

\item $U:\mathcal{E}\rightarrow \underline{\mathcal{C}}$ be the forgetful
functor.
\end{itemize}

Then:

\begin{enumerate}
\item $\left( \mathcal{H},U:\mathcal{E}\rightarrow \underline{\mathcal{C}}%
\right) $ is a capacitor for $\mathcal{C}$;

\item the corresponding negative arrows in $\mathcal{C}$ are all
*-homomorphisms;

\item the functor $U$ creates hulls in the positive monopole of C*-algebras
with respect to vN-dense *-homomorphisms;

\item for a C*-algebra $A$, the corresponding hull is the canonical
inclusion of $A$ into its double dual.
\end{enumerate}
\end{proposition}

\bibliographystyle{amsplain}
\bibliography{bibliography}

@article{marmolejo_kan_2012,
	title = {Kan extensions and lax idempotent pseudomonads},
	volume = {26},
	issn = {1201-561X},
	url = {https://mathscinet.ams.org/mathscinet-getitem?mr=2909637},
	urldate = {2026-03-13},
	journal = {Theory and Applications of Categories},
	author = {Marmolejo, F. and Wood, R. J.},
	year = {2012},
	mrnumber = {2909637},
	pages = {No. 1, 1--29},
}

@article{guitart_monades_1975,
	title = {Monades involutives complémentées},
	volume = {16},
	issn = {0008-0004},
	url = {https://mathscinet.ams.org/mathscinet-getitem?mr=396721},
	number = {1},
	urldate = {2026-03-13},
	journal = {Cahiers de Topologie et Géométrie Différentielle},
	author = {Guitart, René},
	year = {1975},
	mrnumber = {396721},
	pages = {17--101},
}

@article{di_liberti_accessibility_2023,
	title = {Accessibility and presentability in {$2$}-categories},
	volume = {227},
	issn = {0022-4049},
	url = {https://www.sciencedirect.com/science/article/pii/S0022404922001517},
	doi = {10.1016/j.jpaa.2022.107155},
	number = {1},
	urldate = {2026-03-13},
	journal = {Journal of Pure and Applied Algebra},
	author = {Di Liberti, Ivan and Loregian, Fosco},
	month = jan,
	year = {2023},
	pages = {107155},
}

@article{street_yoneda_1978,
	title = {Yoneda structures on {$2$}-categories},
	volume = {50},
	issn = {0021-8693},
	url = {https://www.sciencedirect.com/science/article/pii/0021869378901606},
	doi = {10.1016/0021-8693(78)90160-6},
	number = {2},
	urldate = {2026-03-13},
	journal = {Journal of Algebra},
	author = {Street, Ross and Walters, Robert},
	month = feb,
	year = {1978},
	pages = {350--379},
}

@article{walker_yoneda_2018,
	title = {Yoneda structures and {KZ} doctrines},
	volume = {222},
	issn = {0022-4049},
	url = {https://www.sciencedirect.com/science/article/pii/S0022404917301548},
	doi = {10.1016/j.jpaa.2017.07.004},
	number = {6},
	urldate = {2026-03-13},
	journal = {Journal of Pure and Applied Algebra},
	author = {Walker, Charles},
	month = jun,
	year = {2018},
	pages = {1375--1387},
}

@article{banaschewski_booleanization_1996,
	title = {Booleanization},
	volume = {37},
	issn = {2681-2363},
	url = {https://www.numdam.org/item/?id=CTGDC_1996__37_1_41_0},
	number = {1},
	urldate = {2026-03-09},
	journal = {Cahiers de Topologie et Géométrie Différentielle Catégoriques},
	author = {Banaschewski, Bernhard and Pultr, Aleš},
	year = {1996},
	pages = {41--60},
}

@article{bhattacharjee_hull_2024,
	title = {Hull classes in compact regular frames},
	volume = {85},
	issn = {0002-5240,1420-8911},
	url = {https://mathscinet.ams.org/mathscinet-getitem?mr=4712410},
	doi = {10.1007/s00012-024-00849-5},
	number = {2},
	urldate = {2026-03-09},
	journal = {Algebra Universalis},
	author = {Bhattacharjee, Papiya and Carrera, Ricardo E.},
	year = {2024},
	mrnumber = {4712410},
	pages = {Paper No. 17, 21},
}

@book{blyth_lattices_2005,
	series = {Universitext},
	title = {Lattices and ordered algebraic structures},
	isbn = {978-1-85233-905-0},
	url = {https://mathscinet.ams.org/mathscinet-getitem?mr=2126425},
	urldate = {2026-03-02},
	publisher = {Springer-Verlag London, Ltd., London},
	author = {Blyth, Thomas  S.},
	year = {2005},
	mrnumber = {2126425},
}

@article{cecco_categorical_2024,
	title = {A categorical approach to injective envelopes},
	volume = {15},
	issn = {2639-7390,2008-8752},
	url = {https://mathscinet.ams.org/mathscinet-getitem?mr=4736304},
	doi = {10.1007/s43034-024-00350-z},
	number = {3},
	urldate = {2026-03-07},
	journal = {Annals of Functional Analysis},
	author = {Cecco, Arianna},
	year = {2024},
	mrnumber = {4736304},
	pages = {Paper No. 49, 28},
}

@article{conrad_essential_1971,
	title = {The essential closure of an {Archimedean} lattice-ordered group},
	volume = {38},
	issn = {0012-7094,1547-7398},
	url = {https://mathscinet.ams.org/mathscinet-getitem?mr=277457},
	urldate = {2026-03-02},
	journal = {Duke Mathematical Journal},
	author = {Conrad, Paul},
	year = {1971},
	mrnumber = {277457},
	pages = {151--160},
}

@article{hager_holder_2014,
	title = {Hölder categories},
	volume = {64},
	issn = {0139-9918,1337-2211},
	url = {https://mathscinet.ams.org/mathscinet-getitem?mr=3227761},
	doi = {10.2478/s12175-014-0230-x},
	number = {3},
	urldate = {2026-03-04},
	journal = {Mathematica Slovaca},
	author = {Hager, Anthony W. and Martínez, Jorge},
	year = {2014},
	mrnumber = {3227761},
	pages = {607--642},
}

@article{kubis_injective_2015,
	title = {Injective objects and retracts of {Fraïssé} limits},
	volume = {27},
	issn = {0933-7741,1435-5337},
	url = {https://mathscinet.ams.org/mathscinet-getitem?mr=3334083},
	doi = {10.1515/forum-2012-0081},
	number = {2},
	urldate = {2026-03-07},
	journal = {Forum Mathematicum},
	author = {Kubiś, Wiesław},
	year = {2015},
	mrnumber = {3334083},
	pages = {807--842},
}

@article{bezhanishvili_dedekind_2013,
	title = {Dedekind completions of bounded {Archimedean} {$\ell $}-algebras},
	volume = {12},
	issn = {0219-4988,1793-6829},
	url = {https://mathscinet.ams.org/mathscinet-getitem?mr=3005590},
	doi = {10.1142/S0219498812501393},
	number = {1},
	urldate = {2026-03-07},
	journal = {Journal of Algebra and its Applications},
	author = {Bezhanishvili, Guram and Morandi, Patrick J. and Olberding, Bruce},
	year = {2013},
	mrnumber = {3005590},
	pages = {1250139, 16},
}

@article{bezhanishvili_functorial_2016,
	title = {A functorial approach to {Dedekind} completions and the representation of vector lattices and {$\ell $}-algebras by normal functions},
	volume = {31},
	issn = {1201-561X},
	url = {https://mathscinet.ams.org/mathscinet-getitem?mr=3584699},
	urldate = {2026-03-04},
	journal = {Theory and Applications of Categories},
	author = {Bezhanishvili, Guram and Morandi, Patrick J. and Olberding, Bruce},
	year = {2016},
	mrnumber = {3584699},
	pages = {Paper No. 37, 1095--1133},
}

@article{hager_minimum_2023,
	title = {On minimum proper essential extensions in a category},
	volume = {46},
	issn = {1607-3606,1727-933X},
	url = {https://mathscinet.ams.org/mathscinet-getitem?mr=4576827},
	doi = {10.2989/16073606.2022.2033870},
	number = {3},
	urldate = {2026-03-04},
	journal = {Quaestiones Mathematicae. Journal of the South African Mathematical Society},
	author = {Hager, Anthony W. and Wynne, Brian},
	year = {2023},
	mrnumber = {4576827},
	pages = {495--512},
}

@article{las_vergnas_convexity_1980,
	title = {Convexity in oriented matroids},
	volume = {29},
	issn = {0095-8956,1096-0902},
	url = {https://mathscinet.ams.org/mathscinet-getitem?mr=586435},
	doi = {10.1016/0095-8956(80)90082-9},
	number = {2},
	urldate = {2026-03-04},
	journal = {Journal of Combinatorial Theory. Series B},
	author = {Las Vergnas, Michel},
	year = {1980},
	mrnumber = {586435},
	pages = {231--243},
}

@incollection{bollobas_erdos-ko-rado_1997,
	title = {An {Erd\H{o}s}-{Ko}-{Rado} theorem for signed sets},
	volume = {34},
	issn = {0898-1221,1873-7668},
	url = {https://mathscinet.ams.org/mathscinet-getitem?mr=1486880},
	doi = {10.1016/S0898-1221(97)00215-0},
	number = {11},
	urldate = {2026-03-04},
	booktitle = {Computers \& {Mathematics} with {Applications}. {An} {International} {Journal}},
	author = {Bollobás, Béla and Leader, Imre},
	year = {1997},
	mrnumber = {1486880},
	doi = {10.1016/S0898-1221(97)00215-0},
	note = {Journal Abbreviation: Comput. Math. Appl.},
	pages = {9--13},
}

@article{brini_combinatorics_2005,
	title = {Combinatorics, superalgebras, invariant theory and representation theory.},
	volume = {55},
	issn = {1286-4889},
	url = {https://eudml.org/doc/227556},
	urldate = {2026-03-04},
	journal = {Séminaire Lotharingien de Combinatoire},
	publisher = {Universität Wien, Fakultät für Mathematik},
	author = {Brini, Andrea},
	year = {2005}
}

@article{kainen_weak_1971,
	title = {Weak {Adjoint} {Functors}.},
	volume = {122},
	issn = {0025-5874; 1432-1823},
	url = {https://eudml.org/doc/171575},
	language = {und},
	urldate = {2026-03-04},
	journal = {Mathematische Zeitschrift},
	author = {Kainen, Paul C.},
	year = {1971},
	pages = {1--9},
}

@book{kromer_tool_2007,
	series = {Science {Networks}. {Historical} {Studies}},
	title = {Tool and object},
	volume = {32},
	isbn = {978-3-7643-7523-2},
	url = {https://mathscinet.ams.org/mathscinet-getitem?mr=2272843},
	urldate = {2026-03-02},
	publisher = {Birkhäuser Verlag, Basel},
	author = {Krömer, Ralf},
	year = {2007},
	mrnumber = {2272843},
}

@article{eilenberg_general_1945,
	title = {General theory of natural equivalences},
	volume = {58},
	issn = {0002-9947,1088-6850},
	url = {https://mathscinet.ams.org/mathscinet-getitem?mr=13131},
	doi = {10.2307/1990284},
	urldate = {2026-03-02},
	journal = {Transactions of the American Mathematical Society},
	author = {Eilenberg, Samuel and MacLane, Saunders},
	year = {1945},
	mrnumber = {13131},
	pages = {231--294},
}

@book{borceux_malcev_2004,
	series = {Mathematics and its {Applications}},
	title = {Malcev, protomodular, homological and semi-abelian categories},
	volume = {566},
	isbn = {978-1-4020-1961-6},
	url = {https://mathscinet.ams.org/mathscinet-getitem?mr=2044291},
	doi = {10.1007/978-1-4020-1962-3},
	urldate = {2025-10-31},
	publisher = {Kluwer Academic Publishers, Dordrecht},
	author = {Borceux, Francis and Bourn, Dominique},
	year = {2004},
	mrnumber = {2044291},
}

@article{janelidze_semi-abelian_2002,
	series = {Category {Theory} 1999: selected papers, conference held in {Coimbra} in honour of the 90th birthday of {Saunders} {Mac} {Lane}},
	title = {Semi-abelian categories},
	volume = {168},
	issn = {0022-4049},
	url = {https://www.sciencedirect.com/science/article/pii/S0022404901001037},
	doi = {10.1016/S0022-4049(01)00103-7},
	number = {2},
	urldate = {2025-10-31},
	journal = {Journal of Pure and Applied Algebra},
	author = {Janelidze, George and Márki, László and Tholen, Walter},
	month = mar,
	year = {2002},
	pages = {367--386},
}

@article{raphael_essential_2005,
	title = {On essential ring embeddings and the epimorphic hull of {C}({X})},
	volume = {14},
	url = {https://mathscinet.ams.org/mathscinet-getitem?mr=2122824},
	urldate = {2025-11-02},
	journal = {Theory and Applications of Categories},
	author = {Raphael, Robert  and Woods, R. Grant},
	year = {2005},
	mrnumber = {2122824},
	pages = {No. 2, 46--52},
}

@book{sharpe_injective_1972,
	series = {Cambridge {Tracts} in {Mathematics} and {Mathematical} {Physics}, {No}. 62},
	title = {Injective modules},
	url = {https://mathscinet.ams.org/mathscinet-getitem?mr=360706},
	urldate = {2025-06-14},
	publisher = {Cambridge University Press, London-New York},
	author = {Sharpe, David W. and Vámos, Peter},
	year = {1972},
	mrnumber = {360706},
}

@article{janelidze_ideals_2007,
	title = {Ideals and clots in universal algebra and in semi-abelian categories},
	volume = {307},
	issn = {0021-8693},
	url = {https://www.sciencedirect.com/science/article/pii/S0021869306003711},
	doi = {10.1016/j.jalgebra.2006.05.022},
	number = {1},
	urldate = {2025-11-03},
	journal = {Journal of Algebra},
	author = {Janelidze, George and Márki, László and Ursini, Aldo},
	month = jan,
	year = {2007},
	pages = {191--208},
}

@article{adamek_injective_2002,
	title = {Injective hulls are not natural},
	volume = {48},
	issn = {1420-8911},
	url = {https://doi.org/10.1007/s000120200006},
	doi = {10.1007/s000120200006},
	
	number = {4},
	urldate = {2025-11-04},
	journal = {Algebra Universalis},
	author = {Adámek, Jiří and Herrlich, Horst and Rosický, Jiří and Tholen, Walter},
	month = dec,
	year = {2002},
	pages = {379--388},
}

@incollection{gran_introduction_2021,
	series = {Coimbra {Math}. {Texts}},
	title = {An introduction to regular categories},
	volume = {1},
	url = {https://mathscinet.ams.org/mathscinet-getitem?mr=4367551},
	doi = {10.1007/978-3-030-84319-9_4},
	urldate = {2025-09-28},
	booktitle = {New perspectives in algebra, topology and categories},
	publisher = {Springer, Cham},
	author = {Gran, Marino},
	year = {2021},
	mrnumber = {4367551},
	pages = {113--145},
}

@article{guo_relative_2024,
	title = {Relative injective envelopes and relative projective covers on ring extensions},
	volume = {52},
	issn = {0092-7872},
	url = {https://doi.org/10.1080/00927872.2024.2309525},
	doi = {10.1080/00927872.2024.2309525},
	number = {7},
	urldate = {2026-02-21},
	journal = {Communications in Algebra},
	publisher = {Taylor \& Francis},
	author = {Guo, Shufeng},
	month = jul,
	year = {2024},
	note = {\_eprint: https://doi.org/10.1080/00927872.2024.2309525},
	pages = {2868--2883},
}

@article{day_injectivity_1972,
	title = {Injectivity in {Equational} {Classes} of {Algebras}},
	volume = {24},
	issn = {0008-414X, 1496-4279},
	url = {https://www.cambridge.org/core/journals/canadian-journal-of-mathematics/article/injectivity-in-equational-classes-of-algebras/11F4218736A6568F72BBBEDEBA76C024},
	doi = {10.4153/CJM-1972-017-8},
	
	number = {2},
	urldate = {2026-02-21},
	journal = {Canadian Journal of Mathematics},
	author = {Day, Alan},
	month = apr,
	year = {1972},
	pages = {209--220},
}

@article{porst_characterization_1981,
	title = {Characterization of injective envelopes},
	volume = {22},
	issn = {0008-0004},
	url = {https://mathscinet.ams.org/mathscinet-getitem?mr=639050},
	number = {4},
	urldate = {2026-02-21},
	journal = {Cahiers de Topologie et Géométrie Différentielle},
	author = {Porst, Hans-Eberhard},
	year = {1981},
	mrnumber = {639050},
	pages = {399--406},
}

@article{gran_semi-abelian_2004,
	title = {Semi-abelian monadic categories},
	volume = {13},
	issn = {1201-561X},
	url = {https://mathscinet.ams.org/mathscinet-getitem?mr=2116325},
	urldate = {2026-02-23},
	journal = {Theory and Applications of Categories},
	author = {Gran, Marino and Rosický, Jiří},
	year = {2004},
	mrnumber = {2116325},
	pages = {No. 6, 106--113},
}

@misc{van_daele_multiplier_2025,
	title = {Multiplier algebras and local units},
	url = {http://arxiv.org/abs/2507.08769},
	doi = {10.48550/arXiv.2507.08769},
	urldate = {2026-02-26},
	publisher = {arXiv},
	author = {Van Daele, Alfons and Vercruysse, Joost},
	month = jul,
	year = {2025},
	note = {arXiv:2507.08769},
}

@article{dauns_multiplier_1969,
	title = {Multiplier {Rings} and {Primitive} {Ideals}},
	volume = {145},
	issn = {0002-9947},
	url = {https://www.jstor.org/stable/1995063},
	doi = {10.2307/1995063},
	urldate = {2025-10-20},
	journal = {Transactions of the American Mathematical Society},
	author = {Dauns, John},
	year = {1969},
	note = {Publisher: American Mathematical Society},
	pages = {125--158},
}

@article{lupini_fraisse_2018,
	title = {Fraïssé limits in functional analysis},
	volume = {338},
	issn = {0001-8708},
	url = {https://mathscinet.ams.org/mathscinet-getitem?mr=3861702},
	doi = {10.1016/j.aim.2018.08.012},
	urlyear = {2021-10-11},
	journal = {Advances in Mathematics},
	author = {Lupini, Martino},
	year = {2018},
	mrnumber = {3861702},
	pages = {93--174},
}

@article{bartosova_ramsey_2017,
	title = {The {Ramsey} property for {Banach} spaces and {Choquet} simplices, and applications},
	volume = {355},
	issn = {1631-073X},
	url = {https://mathscinet.ams.org/mathscinet-getitem?mr=3730502},
	doi = {10.1016/j.crma.2017.11.001},
	number = {12},
	urlyear = {2021-10-11},
	journal = {Comptes Rendus Mathématique. Académie des Sciences. Paris},
	author = {Bartošová, Dana and Lopez-Abad, Jordi and Lupini, Martino and Mbombo, Brice},
	year = {2017},
	mrnumber = {3730502},
	pages = {1242--1246},
}

@article{eagle_fraisse_2016,
	title = {Fraïssé limits of {C}*-algebras},
	volume = {81},
	issn = {0022-4812},
	url = {https://mathscinet.ams.org/mathscinet-getitem?mr=3519456},
	doi = {10.1017/jsl.2016.14},
	number = {2},
	urlyear = {2021-10-11},
	journal = {The Journal of Symbolic Logic},
	author = {Eagle, Christopher J. and Farah, Ilijas and Hart, Bradd and Kadets, Boris and Kalashnyk, Vladyslav and Lupini, Martino},
	year = {2016},
	mrnumber = {3519456},
	pages = {755--773},
}

@article{lupini_operator_2015,
	title = {Operator space and operator system analogs of {Kirchberg}'s nuclear embedding theorem},
	volume = {431},
	issn = {0022-247X},
	url = {https://mathscinet.ams.org/mathscinet-getitem?mr=3357573},
	doi = {10.1016/j.jmaa.2015.05.051},
	number = {1},
	urlyear = {2021-10-11},
	journal = {Journal of Mathematical Analysis and Applications},
	author = {Lupini, Martino},
	year = {2015},
	mrnumber = {3357573},
	pages = {47--56},
}

@article{eilenberg_group_1942,
	title = {Group extensions and homology},
	volume = {43},
	issn = {0003-486X},
	url = {https://zbmath.org/?q=an%3A0061.40602},
	doi = {10.2307/1968966},
	urlyear = {2021-10-13},
	journal = {Annals of Mathematics. Second Series},
	author = {Eilenberg, Samuel and MacLane, Saunders},
	year = {1942},
	pages = {757--831},
}

@book{blecher_operator_2004,
	series = {London {Mathematical} {Society} {Monographs}. {New} {Series}},
	title = {Operator algebras and their modules—an operator space approach},
	volume = {30},
	isbn = {978-0-19-852659-9},
	url = {https://mathscinet.ams.org/mathscinet-getitem?mr=2111973},
	urlyear = {2021-12-29},
	publisher = {The Clarendon Press, Oxford University Press, Oxford},
	author = {Blecher, David P. and Le Merdy, Christian},
	year = {2004},
	doi = {10.1093/acprof:oso/9780198526599.001.0001},
	mrnumber = {2111973},
	doi = {10.1093/acprof:oso/9780198526599.001.0001},
}

@book{mac_lane_categories_1998,
	edition = {Second},
	series = {Graduate {Texts} in {Mathematics}},
	title = {Categories for the working mathematician},
	volume = {5},
	isbn = {978-0-387-98403-2},
	url = {https://mathscinet.ams.org/mathscinet-getitem?mr=1712872},
	urlyear = {2021-12-29},
	publisher = {Springer-Verlag, New York},
	author = {Mac Lane, Saunders},
	year = {1998},
	mrnumber = {1712872},
}

@book{awodey_category_2006,
	series = {Oxford {Logic} {Guides}},
	title = {Category theory},
	volume = {49},
	isbn = {978-0-19-856861-2},
	url = {https://mathscinet.ams.org/mathscinet-getitem?mr=2229319},
	urlyear = {2022-02-02},
	publisher = {The Clarendon Press, Oxford University Press, New York},
	author = {Awodey, Steve},
	year = {2006},
	doi = {10.1093/acprof:oso/9780198568612.001.0001},
	mrnumber = {2229319},
	doi = {10.1093/acprof:oso/9780198568612.001.0001},
}

@article{rump_lateral_2009,
	title = {Lateral completion and structure sheaf of an {Archimedean} l-group},
	volume = {213},
	issn = {0022-4049,1873-1376},
	url = {https://mathscinet.ams.org/mathscinet-getitem?mr=2462991},
	doi = {10.1016/j.jpaa.2008.05.013},
	number = {1},
	urldate = {2026-03-02},
	journal = {Journal of Pure and Applied Algebra},
	author = {Rump, Wolfgang and Yang, Yi Chuan},
	year = {2009},
	mrnumber = {2462991},
	pages = {136--143},
}

@article{carrera_hull_2011,
	title = {On hull classes of {$\ell $}-groups and covering classes of spaces},
	volume = {61},
	issn = {0139-9918,1337-2211},
	url = {https://mathscinet.ams.org/mathscinet-getitem?mr=2796252},
	doi = {10.2478/s12175-011-0020-7},
	number = {3},
	urldate = {2026-03-02},
	journal = {Mathematica Slovaca},
	author = {Carrera, Ricardo E. and Hager, Anthony W.},
	year = {2011},
	mrnumber = {2796252},
	pages = {411--428},
}

@article{masulovic_dual_2017,
	title = {A dual {Ramsey} theorem for permutations},
	volume = {24},
	issn = {1077-8926},
	url = {https://mathscinet.ams.org/mathscinet-getitem?mr=3691556},
	doi = {10.37236/6845},
	number = {3},
	urldate = {2026-03-07},
	journal = {Electronic Journal of Combinatorics},
	author = {Masulović, Dragan},
	year = {2017},
	mrnumber = {3691556},
	pages = {Paper No. 3.39, 12},
}

@article{rosicky_uniqueness_2019,
	title = {On the uniqueness of cellular injectives},
	volume = {167},
	issn = {0305-0041,1469-8064},
	url = {https://mathscinet.ams.org/mathscinet-getitem?mr=4015647},
	doi = {10.1017/s0305004118000439},
	number = {3},
	urldate = {2026-03-07},
	journal = {Mathematical Proceedings of the Cambridge Philosophical Society},
	author = {Rosický, Jiří},
	year = {2019},
	mrnumber = {4015647},
	pages = {489--504},
}

@article{ferenczi_amalgamation_2020,
	title = {Amalgamation and {Ramsey} properties of {$L_p$} spaces},
	volume = {369},
	issn = {0001-8708,1090-2082},
	url = {https://mathscinet.ams.org/mathscinet-getitem?mr=4095753},
	doi = {10.1016/j.aim.2020.107190},
	urldate = {2026-03-07},
	journal = {Advances in Mathematics},
	author = {Ferenczi, Valentin and Lopez-Abad, Jordi and Mbombo, Brice and Todorcevic, Stevo},
	year = {2020},
	mrnumber = {4095753},
	pages = {107190, 76},
}

@article{schneiders_quasi-abelian_1999,
	title = {Quasi-abelian categories and sheaves},
	issn = {0249-633X},
	url = {https://mathscinet.ams.org/mathscinet-getitem?mr=1779315},
	number = {76},
	urlyear = {2022-02-13},
	journal = {Mémoires de la Société Mathématique de France. Nouvelle Série},
	author = {Schneiders, Jean-Pierre},
	year = {1999},
	mrnumber = {1779315},
	pages = {vi+134},
}

@book{pedersen_algebras_1979,
	address = {London},
	series = {London {Mathematical} {Society} {Monographs}},
	title = {C*-algebras and their automorphism groups},
	volume = {14},
	isbn = {0-12-549450-5},
	url = {http://www.ams.org/mathscinet-getitem?mr=548006},
	urlyear = {2013-10-31},
	publisher = {Academic Press Inc.},
	author = {Pedersen, Gert K.},
	year = {1979},
}

@book{koppelberg_handbook_1989,
	title = {Handbook of {Boolean} algebras. {Vol}. 1},
	isbn = {0-444-70261-X},
	url = {http://www.ams.org/mathscinet-getitem?mr=991565},
	urlyear = {2014-05-24},
	publisher = {North-Holland Publishing Co., Amsterdam},
	author = {Koppelberg, Sabine},
	year = {1989},
}

@book{hodges_building_1985,
	series = {London {Mathematical} {Society} {Student} {Texts}},
	title = {Building models by games},
	volume = {2},
	isbn = {0-521-26897-4 0-521-31716-9},
	url = {http://www.ams.org/mathscinet-getitem?mr=812274},
	urlyear = {2014-07-10},
	publisher = {Cambridge University Press, Cambridge},
	author = {Hodges, Wilfrid},
	year = {1985},
	mrnumber = {812274},
}

@article{fraisse_sur_1954,
	title = {Sur l'extension aux relations de quelques propriétés des ordres},
	volume = {71},
	issn = {0012-9593},
	url = {http://www.ams.org/mathscinet-getitem?mr=0069239},
	urlyear = {2014-07-04},
	journal = {Annales Scientifiques de l'École Normale Supérieure. Troisième Série},
	author = {Fraïssé, Roland},
	year = {1954},
	pages = {363--388},
}

@book{blackadar_operator_2006,
	address = {Berlin},
	series = {Encyclopaedia of {Mathematical} {Sciences}},
	title = {Operator {Algebras}},
	volume = {122},
	isbn = {978-3-540-28486-4 3-540-28486-9},
	url = {http://www.ams.org/mathscinet-getitem?mr=2188261},
	urlyear = {2013-10-31},
	publisher = {Springer-Verlag},
	author = {Blackadar, Bruce},
	year = {2006},
}

@article{lacey_injective_1969,
	title = {On injective envelopes of {Banach} spaces},
	volume = {4},
	url = {http://www.ams.org/mathscinet-getitem?mr=0243322},
	urlyear = {2014-10-26},
	journal = {J. Functional Analysis},
	author = {Lacey, H. Elton and Cohen, H. B.},
	year = {1969},
	mrnumber = {0243322},
	pages = {11--30},
}

@article{isbell_three_1969,
	title = {Three remarks on injective envelopes of {Banach} spaces},
	volume = {27},
	issn = {0022-247X},
	url = {http://www.sciencedirect.com/science/article/pii/0022247X69901310},
	doi = {10.1016/0022-247X(69)90131-0},
	number = {3},
	urlyear = {2014-10-26},
	journal = {Journal of Mathematical Analysis and Applications},
	author = {Isbell, John R.},
	month = sep,
	year = {1969},
	pages = {516--518},
}

@article{isbell_injective_1964,
	title = {Injective envelopes of {Banach} spaces are rigidly attached},
	volume = {70},
	issn = {0002-9904, 1936-881X},
	url = {http://www.ams.org/bull/1964-70-05/S0002-9904-1964-11192-7/},
	doi = {10.1090/S0002-9904-1964-11192-7},
	number = {5},
	urlyear = {2014-10-26},
	journal = {Bulletin of the American Mathematical Society},
	author = {Isbell, John R.},
	year = {1964},
	pages = {727--729},
}

@article{cohen_injective_1964,
	title = {Injective envelopes of {Banach} spaces},
	volume = {70},
	issn = {0002-9904, 1936-881X},
	url = {http://www.ams.org/bull/1964-70-05/S0002-9904-1964-11189-7/},
	doi = {10.1090/S0002-9904-1964-11189-7},
	number = {5},
	urlyear = {2014-10-26},
	journal = {Bulletin of the American Mathematical Society},
	author = {Cohen, Henry B.},
	year = {1964},
	pages = {723--726},
}

@article{hamana_injective_1979,
	title = {Injective envelopes of operator systems},
	volume = {15},
	issn = {0034-5318},
	url = {http://www.ems-ph.org/journals/show_abstract.php?issn=0034-5318&vol=15&iss=3&rank=8},
	doi = {10.2977/prims/1195187876},
	
	number = {3},
	urlyear = {2014-10-13},
	journal = {Publications of the Research Institute for Mathematical Sciences},
	author = {Hamana, Masamichi},
	year = {1979},
	pages = {773--785},
}

@article{ben_yaacov_linear_2014,
	title = {The linear isometry group of the {Gurarij} space is universal},
	volume = {142},
	issn = {0002-9939, 1088-6826},
	url = {http://www.ams.org/proc/2014-142-07/S0002-9939-2014-11956-3/},
	doi = {10.1090/S0002-9939-2014-11956-3},
	number = {7},
	urlyear = {2014-10-06},
	journal = {Proceedings of the American Mathematical Society},
	author = {Ben Yaacov, Itaï},
	year = {2014},
	pages = {2459--2467},
}

@article{kubis_fraisse_2014,
	title = {Fraïssé sequences: category-theoretic approach to universal homogeneous structures},
	volume = {165},
	issn = {0168-0072},
	shorttitle = {Fraïssé sequences},
	url = {http://www.sciencedirect.com/science/article/pii/S0168007214000773},
	doi = {10.1016/j.apal.2014.07.004},
	number = {11},
	urlyear = {2014-10-07},
	journal = {Annals of Pure and Applied Logic},
	author = {Kubiś, Wiesław},
	year = {2014},
	pages = {1755--1811},
}

@article{choi_injectivity_1977,
	title = {Injectivity and operator spaces},
	volume = {24},
	issn = {0022-1236},
	url = {http://www.sciencedirect.com/science/article/pii/0022123677900520},
	doi = {10.1016/0022-1236(77)90052-0},
	number = {2},
	urlyear = {2014-08-13},
	journal = {Journal of Functional Analysis},
	author = {Choi, Man-Duen and Effros, Edward G.},
	month = feb,
	year = {1977},
	pages = {156--209},
}

@article{ben_yaacov_fraisse_2015,
	title = {Fraïssé limits of metric structures},
	volume = {80},
	issn = {1943-5886},
	url = {http://journals.cambridge.org/article_S0022481214000711},
	doi = {10.1017/jsl.2014.71},
	number = {1},
	urlyear = {2015-05-12},
	journal = {Journal of Symbolic Logic},
	author = {Ben Yaacov, Itaï},
	year = {2015},
	pages = {100--115},
}

@article{hadwin_injectivity_2011,
	title = {Injectivity and projectivity in analysis and topology},
	volume = {54},
	issn = {1674-7283, 1869-1862},
	url = {http://link.springer.com/article/10.1007/s11425-011-4285-7},
	doi = {10.1007/s11425-011-4285-7},
	
	number = {11},
	urlyear = {2015-09-28},
	journal = {Science China Mathematics},
	author = {Hadwin, Don and Paulsen, Vern I.},
	month = aug,
	year = {2011},
	pages = {2347--2359},
}

@article{hamana_injective_1979-1,
	title = {Injective envelopes of {C}*-algebras},
	volume = {31},
	issn = {0025-5645, 1881-1167},
	url = {http://projecteuclid.org.ezproxy.library.yorku.ca/euclid.jmsj/1240319487},
	doi = {10.2969/jmsj/03110181},
	
	number = {1},
	urlyear = {2014-10-13},
	journal = {Journal of the Mathematical Society of Japan},
	author = {Hamana, Masamichi},
	month = jan,
	year = {1979},
	mrnumber = {MR0519044},
	zmnumber = {0395.46042},
	pages = {181--197},
}

@incollection{hamana_injective_1992,
	address = {Harlow},
	series = {Pitman {Research} {Notes} in {Mathematics} {Series}},
	title = {Injective envelopes of dynamical systems},
	volume = {271},
	booktitle = {Operator algebras and operator theory},
	publisher = {Longman Sci. Tech.},
	author = {Hamana, Masamichi},
	year = {1992},
}

@article{hamana_injective_1978,
	title = {Injective envelopes of {Banach} modules},
	volume = {30},
	issn = {0040-8735},
	url = {http://projecteuclid.org/euclid.tmj/1178229979},
	doi = {10.2748/tmj/1178229979},
	
	number = {3},
	urlyear = {2015-09-28},
	journal = {Tohoku Mathematical Journal},
	author = {Hamana, Masamichi},
	year = {1978},
	pages = {439--453},
}

@article{hamana_injective_1985,
	title = {Injective envelopes of {C}*-dynamical systems},
	volume = {37},
	issn = {0040-8735},
	url = {http://projecteuclid.org/euclid.tmj/1178228589},
	doi = {10.2748/tmj/1178228589},
	
	number = {4},
	urlyear = {2015-09-28},
	journal = {Tohoku Mathematical Journal},
	author = {Hamana, Masamichi},
	year = {1985},
	pages = {463--487},
}

@article{hamana_injective_2011,
	title = {Injective envelopes of dynamical systems},
	volume = {34},
	issn = {1880-6015},
	url = {http://www.ams.org/mathscinet-getitem?mr=2985658},
	urlyear = {2015-09-28},
	journal = {Toyama Mathematical Journal},
	author = {Hamana, Masamichi},
	year = {2011},
	pages = {23--86},
}

@article{lang_injective_2013,
	title = {Injective hulls of certain discrete metric spaces and groups},
	volume = {5},
	issn = {1793-5253},
	url = {http://www.ams.org/mathscinet-getitem?mr=3096307},
	doi = {10.1142/S1793525313500118},
	number = {3},
	urlyear = {2016-01-07},
	journal = {Journal of Topology and Analysis},
	author = {Lang, Urs},
	year = {2013},
	mrnumber = {3096307},
	pages = {297--331},
}

@book{hodges_model_1993,
	series = {Encyclopedia of {Mathematics} and its {Applications}},
	title = {Model theory},
	volume = {42},
	isbn = {0-521-30442-3},
	url = {http://www.ams.org/mathscinet-getitem?mr=1221741},
	urlyear = {2014-11-10},
	publisher = {Cambridge University Press, Cambridge},
	author = {Hodges, Wilfrid},
	year = {1993},
}

@article{masulovic_pre-adjunctions_2018,
	title = {Pre-adjunctions and the {Ramsey} property},
	volume = {70},
	issn = {0195-6698,1095-9971},
	url = {https://mathscinet.ams.org/mathscinet-getitem?mr=3779618},
	doi = {10.1016/j.ejc.2018.01.006},
	urldate = {2026-03-07},
	journal = {European Journal of Combinatorics},
	author = {Masulović, Dragan},
	year = {2018},
	mrnumber = {3779618},
	pages = {268--283},
}

@book{timmermann_invitation_2008,
	series = {{EMS} {Textbooks} in {Mathematics}},
	title = {An invitation to quantum groups and duality},
	isbn = {978-3-03719-043-2},
	url = {http://www.ams.org/mathscinet-getitem?mr=2397671},
	urlyear = {2014-04-10},
	publisher = {European Mathematical Society},
	author = {Timmermann, Thomas},
	year = {2008},
}

@article{van_daele_multiplier_1994,
	title = {Multiplier {Hopf} {Algebras}},
	volume = {342},
	copyright = {Copyright © 1994 American Mathematical Society},
	url = {http://www.jstor.org/stable/2154659},
	doi = {10.2307/2154659},
	number = {2},
	urlyear = {2014-04-08},
	journal = {Transactions of the American Mathematical Society},
	author = {Van Daele, Alfons},
	month = apr,
	year = {1994},
	pages = {917--932},
}

@article{masumoto_jiang-su_2017,
	title = {The {Jiang}-{Su} algebra as a {Fraïssé} limit},
	volume = {82},
	issn = {0022-4812,1943-5886},
	url = {https://mathscinet.ams.org/mathscinet-getitem?mr=3743622},
	doi = {10.1017/jsl.2016.52},
	number = {4},
	urldate = {2026-03-07},
	journal = {The Journal of Symbolic Logic},
	author = {Masumoto, Shuhei},
	year = {2017},
	mrnumber = {3743622},
	pages = {1541--1559},
}

@article{kubis_lelek_2017,
	title = {The {Lelek} fan and the {Poulsen} simplex as {Fraïssé} limits},
	volume = {111},
	issn = {1578-7303},
	url = {https://mathscinet.ams.org/mathscinet-getitem?mr=3690070},
	doi = {10.1007/s13398-016-0339-6},
	number = {4},
	urlyear = {2018-08-22},
	journal = {Revista de la Real Academia de Ciencias Exactas, F{\textbackslash}textbackslash'ı sicas y Naturales. Serie A. Matematicas. RACSAM},
	author = {Kubiś, Wiesław and Kwiatkowska, Aleksandra},
	year = {2017},
	mrnumber = {3690070},
	pages = {967--981},
}

@book{lurie_higher_2009,
	series = {Annals of {Mathematics} {Studies}},
	title = {Higher topos theory},
	volume = {170},
	isbn = {978-0-691-14049-0},
	url = {https://mathscinet.ams.org/mathscinet-getitem?mr=2522659},
	urlyear = {2019-04-14},
	publisher = {Princeton University Press, Princeton, NJ},
	author = {Lurie, Jacob},
	year = {2009},
	mrnumber = {2522659},
	doi = {10.1515/9781400830558},
}

@article{busby_double_1968,
	title = {Double centralizers and extensions of {C}*-algebras},
	volume = {132},
	issn = {0002-9947},
	url = {https://mathscinet.ams.org/mathscinet-getitem?mr=225175},
	doi = {10.2307/1994883},
	urlyear = {2020-04-18},
	journal = {Transactions of the American Mathematical Society},
	author = {Busby, Robert C.},
	year = {1968},
	mrnumber = {225175},
	pages = {79--99},
}

@book{lance_hilbert_1995,
	series = {London {Mathematical} {Society} {Lecture} {Note} {Series}},
	title = {Hilbert {C}*-modules},
	volume = {210},
	isbn = {978-0-521-47910-3},
	url = {https://mathscinet.ams.org/mathscinet-getitem?mr=1325694},
	urlyear = {2022-05-15},
	publisher = {Cambridge University Press, Cambridge},
	author = {Lance, E. Christopher},
	year = {1995},
	mrnumber = {1325694},
	doi = {10.1017/CBO9780511526206},
}

@article{kelly_basic_2005,
	title = {Basic concepts of enriched category theory},
	number = {10},
	journal = {Reprints in Theory and Applications of Categories},
	author = {Kelly, Gregory  M.},
	year = {2005},
	mrnumber = {2177301},
	pages = {vi+137},
}

@article{emmanouil_flat_2011,
	title = {On the flat length of injective modules},
	volume = {84},
	copyright = {© 2011 London Mathematical Society},
	issn = {1469-7750},
	doi = {10.1112/jlms/jdr014},
	
	number = {2},
	urlyear = {2024-12-08},
	journal = {Journal of the London Mathematical Society},
	author = {Emmanouil, Ioannis and Talelli, Olympia},
	year = {2011}
}

@article{facchini_generalized_1994,
	title = {Generalized {Dedekind} domains and their injective modules},
	volume = {94},
	issn = {0022-4049},
	url = {https://www.sciencedirect.com/science/article/pii/0022404994900302},
	doi = {10.1016/0022-4049(94)90030-2},
	number = {2},
	urldate = {2025-06-14},
	journal = {Journal of Pure and Applied Algebra},
	author = {Facchini, Alberto},
	month = jun,
	year = {1994},
	pages = {159--173},
}

@article{matlis_injective_1959,
	title = {Injective modules over {Prüfer} rings},
	volume = {15},
	issn = {0027-7630,2152-6842},
	journal = {Nagoya Mathematical Journal},
	author = {Matlis, Eben},
	year = {1959},
	mrnumber = {109840},
	pages = {57--69},
}

@article{ruan_injectivity_1989,
	title = {Injectivity of operator spaces},
	volume = {315},
	issn = {0002-9947,1088-6850},
	url = {https://mathscinet.ams.org/mathscinet-getitem?mr=929239},
	doi = {10.2307/2001374},
	number = {1},
	urldate = {2026-03-02},
	journal = {Transactions of the American Mathematical Society},
	author = {Ruan, Zhong-Jin},
	year = {1989},
	mrnumber = {929239},
	pages = {89--104},
}

@article{raphael_epimorphic_2000,
	title = {The epimorphic hull of {C}({X})},
	volume = {105},
	issn = {0166-8641},
	url = {https://www.sciencedirect.com/science/article/pii/S016686419900036X},
	doi = {10.1016/S0166-8641(99)00036-X},
	number = {1},
	urldate = {2026-03-02},
	journal = {Topology and its Applications},
	author = {Raphael, Robert M. and Woods, R. Grant},
	month = jul,
	year = {2000},
	pages = {65--88},
}

@article{oikhberg_injectivity_2018,
	title = {Injectivity and projectivity in {$p$}-multinormed spaces},
	volume = {22},
	issn = {1385-1292,1572-9281},
	url = {https://mathscinet.ams.org/mathscinet-getitem?mr=3843454},
	doi = {10.1007/s11117-018-0557-6},
	number = {4},
	urldate = {2026-03-02},
	journal = {Positivity. },
	author = {Oikhberg, Timur},
	year = {2018},
	mrnumber = {3843454},
	pages = {1023--1037},
}

@phdthesis{bryder_boundaries_2017,
	address = {Copenhagen, Denmark},
	type = {{PhD} {Thesis}},
	title = {Boundaries, {Injective} {Envelopes}, and {Reduced} {Crossed} {Producs}},
	school = {University of Copenhagen},
	author = {Bryder, Rasmus Sylverster},
	year = {2017},
}

@incollection{martinez_hull_2002,
	series = {Dev. {Math}.},
	title = {Hull classes of {Archimedean} lattice-ordered groups with unit: a survey},
	volume = {7},
	isbn = {978-1-4020-0752-1},
	shorttitle = {Hull classes of {Archimedean} lattice-ordered groups with unit},
	url = {https://mathscinet.ams.org/mathscinet-getitem?mr=2083035},
	urldate = {2026-03-02},
	booktitle = {Ordered algebraic structures},
	publisher = {Kluwer Acad. Publ., Dordrecht},
	author = {Martínez, Jorge},
	year = {2002},
	mrnumber = {2083035},
	pages = {89--121},
}

@book{heunen_categories_2019,
	series = {Oxford {Graduate} {Texts} in {Mathematics}},
	title = {Categories for quantum theory},
	volume = {28},
	isbn = {978-0-19-873961-6 978-0-19-873962-3},
	url = {https://mathscinet.ams.org/mathscinet-getitem?mr=3971584},
	doi = {10.1093/oso/9780198739623.001.0001},
	urldate = {2026-03-03},
	publisher = {Oxford University Press, Oxford},
	author = {Heunen, Chris and Vicary, Jamie},
	year = {2019},
	mrnumber = {3971584},
	doi = {10.1093/oso/9780198739623.001.0001},
}

@article{bernau_unique_1965,
	title = {Unique representation of {Archimedean} lattice groups and normal {Archimedean} lattice rings},
	volume = {15},
	issn = {0024-6115,1460-244X},
	url = {https://mathscinet.ams.org/mathscinet-getitem?mr=182661},
	doi = {10.1112/plms/s3-15.1.599},
	urldate = {2026-03-05},
	journal = {Proceedings of the London Mathematical Society. Third Series},
	author = {Bernau, Simon J.},
	year = {1965},
	mrnumber = {182661},
	pages = {599--631},
}

@article{conrad_lateral_1969,
	title = {The lateral completion of a lattice-ordered group},
	volume = {19},
	issn = {0024-6115,1460-244X},
	url = {https://mathscinet.ams.org/mathscinet-getitem?mr=244125},
	doi = {10.1112/plms/s3-19.3.444},
	urldate = {2026-03-07},
	journal = {Proceedings of the London Mathematical Society. Third Series},
	author = {Conrad, Paul},
	year = {1969},
	mrnumber = {244125},
	pages = {444--480},
}

@book{davey_introduction_1990,
	series = {Cambridge {Mathematical} {Textbooks}},
	title = {Introduction to lattices and order},
	isbn = {978-0-521-36584-0 978-0-521-36766-0},
	url = {https://mathscinet.ams.org/mathscinet-getitem?mr=1058437},
	urldate = {2026-03-05},
	publisher = {Cambridge University Press, Cambridge},
	author = {Davey, Brian  A. and Priestley, Hilary  A.},
	year = {1990},
	mrnumber = {1058437},
}

@article{banaschewski_categorical_1967,
	title = {Categorical characterization of the {MacNeille} completion},
	volume = {18},
	issn = {0003-889X,1420-8938},
	url = {https://mathscinet.ams.org/mathscinet-getitem?mr=221984},
	doi = {10.1007/BF01898828},
	urldate = {2026-03-05},
	journal = {Archiv der Mathematik},
	author = {Banaschewski, Bernhard and Bruns, Günter W.},
	year = {1967},
	mrnumber = {221984},
	pages = {369--377},
}

@article{banaschewski_strong_2010,
	title = {On the strong amalgamation of {Boolean} algebras},
	volume = {63},
	issn = {1420-8911},
	url = {https://doi.org/10.1007/s00012-010-0072-5},
	doi = {10.1007/s00012-010-0072-5},
	
	number = {2},
	urldate = {2026-03-06},
	journal = {Algebra Universalis},
	author = {Banaschewski, Bernhard},
	month = may,
	year = {2010},
	pages = {235--238},
}

@book{sikorski_boolean_1969,
	edition = {Third},
	series = {Ergebnisse der {Mathematik} und ihrer {Grenzgebiete}, {Band} 25},
	title = {Boolean algebras},
	url = {https://mathscinet.ams.org/mathscinet-getitem?mr=242724},
	urldate = {2026-03-06},
	publisher = {Springer-Verlag New York, Inc., New York},
	author = {Sikorski, Roman},
	year = {1969},
	mrnumber = {242724},
}

@inproceedings{halmos_injective_1961,
	title = {Injective and projective {Boolean} algebras},
	url = {https://mathscinet.ams.org/mathscinet-getitem?mr=137671},
	urldate = {2026-03-06},
	booktitle = {Proc. {Sympos}. {Pure} {Math}., {Vol}. {II}},
	publisher = {Amer. Math. Soc., Providence, RI},
	author = {Halmos, Paul R.},
	year = {1961},
	mrnumber = {137671},
	pages = {114--122},
}

@book{schroder_ordered_2016,
	edition = {Second},
	title = {Ordered sets},
	isbn = {978-3-319-29786-6 978-3-319-29788-0},
	url = {https://mathscinet.ams.org/mathscinet-getitem?mr=3469976},
	doi = {10.1007/978-3-319-29788-0},
	urldate = {2026-03-07},
	publisher = {Birkhäuser/Springer},
	author = {Schröder, Bernd},
	year = {2016},
	mrnumber = {3469976},
	doi = {10.1007/978-3-319-29788-0},
}

@book{rudeanu_sets_2012,
	title = {Sets and ordered structures},
	isbn = {978-1-60805-338-4},
	url = {https://mathscinet.ams.org/mathscinet-getitem?mr=3478116},
	doi = {10.2174/978160805338411201010001},
	urldate = {2026-03-07},
	publisher = {Bentham Science Publishers, Ltd., Oak Park, IL},
	author = {Rudeanu, Sergiu},
	year = {2012},
	mrnumber = {3478116},
	doi = {10.2174/978160805338411201010001},
}

@book{harzheim_ordered_2005,
	series = {Advances in {Mathematics} ({Springer})},
	title = {Ordered sets},
	volume = {7},
	isbn = {978-0-387-24219-4},
	url = {https://mathscinet.ams.org/mathscinet-getitem?mr=2127991},
	urldate = {2026-03-07},
	publisher = {Springer, New York},
	author = {Harzheim, Egbert},
	year = {2005},
	mrnumber = {2127991},
}

@article{cantier_fraisse_2024,
	title = {Fraïssé theory for {Cuntz} semigroups},
	volume = {658},
	issn = {0021-8693},
	url = {https://www.sciencedirect.com/science/article/pii/S0021869324003363},
	doi = {10.1016/j.jalgebra.2024.05.052},
	urldate = {2026-03-07},
	journal = {Journal of Algebra},
	author = {Cantier, Laurent and Vilalta, Eduard},
	month = nov,
	year = {2024},
	pages = {319--364},
}

@article{tursi_separable_2023,
	title = {A separable universal homogeneous {Banach} lattice},
	issn = {1073-7928,1687-0247},
	url = {https://mathscinet.ams.org/mathscinet-getitem?mr=4565693},
	doi = {10.1093/imrn/rnac024},
	number = {7},
	urldate = {2026-03-07},
	journal = {International Mathematics Research Notices.},
	author = {Tursi, Mary Angelica},
	year = {2023},
	mrnumber = {4565693},
	pages = {5438--5472},
}

@article{bryant_fraisse_2021,
	title = {Fraïssé limits for relational metric structures},
	volume = {86},
	issn = {0022-4812,1943-5886},
	url = {https://mathscinet.ams.org/mathscinet-getitem?mr=4347563},
	doi = {10.1017/jsl.2021.65},
	number = {3},
	urldate = {2026-03-07},
	journal = {The Journal of Symbolic Logic},
	author = {Bryant, David and Nies, André and Tupper, Paul},
	year = {2021},
	mrnumber = {4347563},
	pages = {913--934},
}

@article{ghasemi_strongly_2021,
	title = {Strongly self-absorbing {C}*-algebras and {Fraïssé} limits},
	volume = {53},
	issn = {0024-6093,1469-2120},
	url = {https://mathscinet.ams.org/mathscinet-getitem?mr=4275102},
	doi = {10.1112/blms.12474},
	number = {3},
	urldate = {2026-03-07},
	journal = {Bulletin of the London Mathematical Society},
	author = {Ghasemi, Saeed},
	year = {2021},
	mrnumber = {4275102},
	pages = {937--955},
}

@article{ghasemi_universal_2020,
	title = {Universal {AF}-algebras},
	volume = {279},
	issn = {0022-1236,1096-0783},
	url = {https://mathscinet.ams.org/mathscinet-getitem?mr=4097283},
	doi = {10.1016/j.jfa.2020.108590},
	number = {5},
	urldate = {2026-03-07},
	journal = {Journal of Functional Analysis},
	author = {Ghasemi, Saeed and Kubiś, Wiesław},
	year = {2020},
	mrnumber = {4097283},
	pages = {108590, 32},
}

@article{kubis_game-theoretic_2018,
	title = {Game-theoretic characterization of the {Gurarii} space},
	volume = {110},
	issn = {0003-889X,1420-8938},
	url = {https://mathscinet.ams.org/mathscinet-getitem?mr=3742292},
	doi = {10.1007/s00013-017-1088-2},
	number = {1},
	urldate = {2026-03-07},
	journal = {Archiv der Mathematik},
	author = {Kubiś, Wiesław},
	year = {2018},
	mrnumber = {3742292},
	pages = {53--59},
}

@article{kakol_non-archimedean_2017,
	title = {On non-archimedean {Gurarii} spaces},
	volume = {450},
	issn = {0022-247X,1096-0813},
	url = {https://mathscinet.ams.org/mathscinet-getitem?mr=3639084},
	doi = {10.1016/j.jmaa.2017.01.072},
	number = {2},
	urldate = {2026-03-07},
	journal = {Journal of Mathematical Analysis and Applications},
	author = {K\c{a}kol, Jerzy and Kubiś, Wiesław and Kubzdela, Albert},
	year = {2017},
	mrnumber = {3639084},
	pages = {969--981},
}

@article{macneille_partially_1937,
	title = {Partially ordered sets},
	volume = {42},
	issn = {0002-9947,1088-6850},
	url = {https://mathscinet.ams.org/mathscinet-getitem?mr=1501929},
	doi = {10.2307/1989739},
	number = {3},
	urldate = {2026-03-07},
	journal = {Transactions of the American Mathematical Society},
	author = {MacNeille, Holbrook  M.},
	year = {1937},
	mrnumber = {1501929},
	pages = {416--460},
}

@article{rump_essential_2009,
	title = {The essential cover and the absolute cover of a schematic space},
	volume = {114},
	issn = {0010-1354,1730-6302},
	url = {https://mathscinet.ams.org/mathscinet-getitem?mr=2457279},
	doi = {10.4064/cm114-1-6},
	number = {1},
	urldate = {2026-03-02},
	journal = {Colloquium Mathematicum},
	author = {Rump, Wolfgang and Yang, Yi Chuan},
	year = {2009},
	mrnumber = {2457279},
	pages = {53--75},
}

@article{hager_hulls_1999,
	title = {Hulls for various kinds of {$\alpha $}-completeness in {Archimedean} lattice-ordered groups},
	volume = {16},
	issn = {0167-8094,1572-9273},
	url = {https://mathscinet.ams.org/mathscinet-getitem?mr=1740743},
	doi = {10.1023/A:1006323031986},
	number = {1},
	urldate = {2026-03-04},
	journal = {Order.},
	author = {Hager, Anthony W. and Martinez, Jorge},
	year = {1999},
	mrnumber = {1740743},
	pages = {89--103 (2000)},
}

@article{anderson_essential_1979,
	title = {The essential closure of {C}({X})},
	volume = {76},
	issn = {0002-9939,1088-6826},
	url = {https://mathscinet.ams.org/mathscinet-getitem?mr=534378},
	doi = {10.2307/2042905},
	number = {1},
	urldate = {2026-03-02},
	journal = {Proceedings of the American Mathematical Society},
	author = {Anderson, Marlow},
	year = {1979},
	mrnumber = {534378},
	pages = {8--10},
}

@article{bernau_lateral_1975,
	title = {The lateral completion of an arbitrary lattice group},
	volume = {19},
	url = {https://mathscinet.ams.org/mathscinet-getitem?mr=384640},
	urldate = {2026-03-05},
	journal = {Journal of the Australian Mathematical Society},
	author = {Bernau, Simon J.},
	year = {1975},
	mrnumber = {384640},
	pages = {263--289},
}

@article{chen_amalgamable_2019,
	title = {Amalgamable diagram shapes},
	volume = {84},
	issn = {0022-4812,1943-5886},
	url = {https://mathscinet.ams.org/mathscinet-getitem?mr=3922786},
	doi = {10.1017/jsl.2018.87},
	number = {1},
	urldate = {2026-03-07},
	journal = {The Journal of Symbolic Logic},
	author = {Chen, Ruiyuan},
	year = {2019},
	mrnumber = {3922786},
	pages = {88--101},
}

@article{masulovic_categorical_2017,
	title = {Categorical equivalence and the {Ramsey} property for finite powers of a primal algebra},
	volume = {78},
	issn = {0002-5240,1420-8911},
	url = {https://mathscinet.ams.org/mathscinet-getitem?mr=3697187},
	doi = {10.1007/s00012-017-0453-0},
	number = {2},
	urldate = {2026-03-07},
	journal = {Algebra Universalis},
	author = {Masulović, Dragan and Scow, Lynn},
	year = {2017},
	mrnumber = {3697187},
	pages = {159--179},
}

\end{document}